\documentclass{amsart}

\usepackage{ifpdf}
\usepackage{ae,aecompl}\usepackage{ifthen}
\usepackage[a4paper,twoside]{geometry}
\usepackage{amsthm}
\usepackage{graphicx}
\usepackage{amssymb}
\usepackage{url}

\newboolean{dum}
\setboolean{dum}{false}
\newcommand{\comment}[1]{\ifthenelse{\boolean{dum}}{
{\par\noindent\Huge\ding{46}} \fbox{\parbox{10cm}{#1}}\par}{}}
\newcommand{\clabel}[1]{\comment{label #1}}

\newtheorem{lemma}[subsection]{Lemma}
\newtheorem{thm}[subsection]{Theorem}
\newtheorem{prop}[subsection]{Proposition}
\newtheorem{cor}[subsection]{Corollary}
\numberwithin{equation}{section}

\DeclareMathOperator{\Real}{Re}

\begin{document}
\title{A finitization of the bead process}
\author{Benjamin J. Fleming}
\address{Department of Mathematics and Statistics \\
The University of Melbourne 
Victoria 3010, Australia}
\email{B.Fleming2@pgrad.unimelb.edu.au}

\author{Peter J. Forrester}
\address{Department of Mathematics and Statistics \\
The University of Melbourne 
Victoria 3010, Australia}
\email{P.Forrester@ms.unimelb.edu.au}

\author{Eric Nordenstam}
\address{Institutionen f\"or Matematik, 
Swedish Royal Institute of Technology (KTH), 100 44 Stockholm, Sweden}
\email{eno@math.kth.se}

\begin{abstract}
The bead process is the particle system defined on parallel lines, with underlying measure giving
constant weight to all configurations in which particles on neighbouring lines interlace,
and zero weight otherwise. Motivated by the statistical mechanical model of the tiling of
an $abc$-hexagon by three species of rhombi, a finitized version of the bead process is defined.
The corresponding joint distribution can be realized as an eigenvalue probability density function
for a sequence of random matrices.
The finitized bead process is determinantal, and we give the correlation kernel
in terms of Jacobi polynomials. Two scaling limits are considered: a global limit in which the
spacing between lines goes to zero, and a certain bulk scaling limit.
In the global limit the shape of the
support of the particles is determined, while in the bulk scaling limit the bead process kernel
of Boutillier is reclaimed, after approriate identification of the anisotropy parameter therein.
\end{abstract}

\maketitle

\section{Introduction}
Consider a particle system confined to equally spaced lines that are parallel to
the $y$-axis and pass through the points $x=k$ for $k \in \mathbb Z^+$, and let this
latter integer label the lines. For each line $k$ place point particles uniformly
at random, with the constraint that their positions interlace with the positions
of the particles on lines $k \pm 1$. It has been shown by Boutillier \cite{Bo06}
that such uniformly distributed interlaced configurations --- referred to as a bead
process --- are naturally generated within dimer models. Furthermore, it is shown
in  \cite{Bo06} that the bead process is an example of a determinantal point
process: the $k$-point correlation function is equal to a $k \times k$ determinant
with entries independent of $k$. If all the particles are on the same line, this
correlation function is precisely that for the eigenvalues in the bulk of a random
matrix ensemble with unitary symmetry, for example the Gaussian unitary ensemble of
complex Hermitian matrices (GUE).

If some of the particles are on different lines, and the 
anisotropy parameter (see below for this notion) is
zero, the correlations between particles on different
lines coincide with the correlation between bulk eigenvalues in the 
GUE minor process \cite{FN08}. The latter is the multi-species particle
system formed by the eigenvalues of a GUE matrix and its successive minors.
Let $M_{n}$ denote an $n \times n$ GUE matrix, and form the principal minors
$M_t$, $t=1,\dots,n-1$
as the  $t \times t$ GUE matrices corresponding to the top $t \times t$ 
block of $M_{n}$. Let the eigenvalues of $M_t$ be denoted
by $\{x_j^{(t)}\}_{j=1,\dots,t}$ with $x_t^{(t)} < \cdots < x_1^{(t)}$.
By a theorem of linear algebra, the successive minors $M_{t+1}$, $M_t$ have
 the fundamental property that
their eigenvalues interlace, 
\begin{equation}\label{II}
x_{t+1}^{(t+1)} < x_t^{(t)} < x_t^{(t+1)} < \cdots < x_2^{(t+1)} < x_1^{(t)} < x_1^{(t+1)}
\end{equation}
for $t=1, \dots, n-1$. It turns out that for
$n \times n$ GUE matrices \cite{Ba01a}, the joint distribution of all 
the $\frac{1}{2}n(n+1)$ eigenvalues $\cup_{s=1}^{n} \{ x_j^{(s)} \}_{j=1,\dots,s}$ 
of all successive minors is proportional to
\begin{equation}\label{1.2}
\prod_{s=1}^{n-1} \chi(x^{(s)},x^{(s+1)} ) \prod_{j=1}^{n} e^{- (x_j^{(n)} )^2}
\prod_{1 \le j < k \le n } (x_j^{(n)} - x_k^{(n)} )
\end{equation}
where $\chi(x^{(s)}, x^{(s+1)} )$ denotes the interlacing (\ref{II}) between
neighbouring species. Regarding species $(s)$ as occuring on line $s$, one
sees that the eigenvalues to the left of line $n$ all occur uniformly within interlaced
regions, precisely as required by the bead process.

The main objective of this paper is to construct and analyze a finitization of the
bead process. Like in \cite{Bo06}, we are motivated by a statistical mechanical
model --- in our case the tiling of an $abc$-hexagon by three species of rhombi.
However unlike in  \cite{Bo06} this statistical mechanical model plays no
essential role in the ensuing analysis. We will see in Section 2 that the rhombi tiling
of an $abc$-hexagon leads to a particle system defined on the segments $0 < y < 1$ of the lines at
$x=1,2,\dots,p+q-1$ $(p \le q)$ in the $xy$-plane. The number of particles, $r(t)$ say, 
on line $x=t$ (to be referred henceforth as line $t$) is
given by
\begin{equation}\label{3.1}
r(t) = \left \{
\begin{array}{ll} t, & t \le p \\
p, & p \le t \leq q \\
p+q-t, & q \le t.
\end{array} \right.
\end{equation}
Denote the coordinates on line $t$ by $x^{(t)} := \{ x_j^{(t)} \}_{j=1,\dots, r(t)}$, and require that
the particles be ordered $x_{r(t)} < \cdots < x_1$. With $\chi(x^{(s)}, x^{(s+1)})$ denoting the interlacing
(\ref{II}) between particles on neighbouring lines the particle system---which is our finitization of the
bead process---is specified by the joint probability density function (PDF) on the line segments
proportional to
\begin{equation}\label{3.2}
\prod_{t=1}^{p-1} \chi(x^{(t)}, x^{(t+1)})
\prod_{t=p}^{q-1} \chi(x^{(t)}, x^{(t+1)} \cup \{0\})
\prod_{t=q}^{p+q-1} \chi(x^{(t)}, x^{(t+1)} \cup \{0,1\}).
\end{equation}

The particles in the tiling problem are restricted to lattice sites. To obtain our
finitization of the bead process requires taking a continuum limit in which the
lengths of the vertical side of the $abc$-hexagon is taken to infinity. This is done in
Section 3. There we also compute the marginal distribution of (\ref{3.2}) on any one line,
and we use that knowledge, together with some random matrix theory, to compute the shape
formed by the support of the particles in the limit that the lines form a continuum in the
horizontal direction.

We know that the PDF (\ref{3.2}) can be realized as a joint eigenvalue PDF as well as
a distribution relating to both queues \cite{Ba01a} and tilings \cite{JN06}. Likewise,
in Section 4 we show that the PDF (\ref{3.2}), which can be obtained as the continuum limit
of a tiling problem, can also be realized as an eigenvalue PDF for a certain sequence of
random matrices.

Section 5 is devoted to the computation of the general $n$-point correlation function for
$n$ particles on arbitrary lines in our finitized bead process. The joint PDF (\ref{3.2}) is
a determinantal point process, meaning that the $n$-point correlation function can be
specified as an $n \times n$ determinant with entries independent of $n$. To compute the
entries, referred to as the correlation kernel, we use a method based on a formulation in
terms of the diagonalization of certain matrices, due to Borodin and collaborators
\cite{BFPS06}. In Section 6 we take up the task of computing the bulk scaling of the
correlation function. This involves us choosing the origin at the midpoint of a collection
lines spaced $O(1)$ apart, scaling the distances on the lines so that the mean density in the
neigbourhood of these points is unity, then taking the limit that the number of particles
goes to infinity.
At a technical level, this requires use of a certain asymptotic formula
for Jacobi polynomials, due to Chen and Ismail~\cite{CI91}.
Their formula contains small errors of detail in its presentation, which 
 we remedy in the Appendix.
 We reclaim the bead process correlation kernel obtained by
Boutillier \cite{Bo06}, after identification of the anisotropy parameter therein with a
parameter defined in terms of the continuum line number, and an anisotropy parameter
relating to our underlying $abc$-hexagon.

\section{Discrete model}

Construct a square grid from the lines $x = j$, $y =  k$, ($j,k \in \mathbb Z$) in the $xy$-plane.
On this grid construct continuous lattice paths which consist of straight line segments connecting the
lattice point $(x,y)$ to $(x+1, y \pm 1)$ for one of $\pm$. Such a lattice path, starting at $(x,y)$
and finishing at $(x+N,y+p)$ say, can be thought of as the space time trajectory of a random walker
confined to the $y$-axis. The random walker starts at $y$, and finishes at $y+p$ after $N$ steps,
each of which are one unit down, or one unit up. We are interested in $n$ such lattice paths,
conditioned never to intersect,
starting at $(0,2i)$ and ending at $(p+q, -p+q+2i)$
for $i=0$, \dots, $n-1$.
Such configurations are in bijection with lozenge tilings
of a hexagon as seen in Figure~\ref{fig:hexa}.
It can be shown,~\cite{Jo02} that the positions
of the lattice paths along the lines, i.e.~the
red particles in Figure~\ref{fig:hexa}, form a determinantal
process. But our interest is not in the particle process defined by the paths, which is
further studied in~\cite{OR06,Go07} but rather the
process defined by the empty sites (holes) between the paths, 
which in Figure~\ref{fig:hexa} are
the blue particles.
It suffices to assume 
$p\leq q$, as the other case reduces to this under 
reflection.

Along line $t$ there will be $r(t)$ blue particles, where $r(t)$ is given by (\ref{3.1}).
Let $\{x_j^{(t)}\}_{j=1,\dots,r(t)}$, with $x_{r(t)} < \cdots < x_1$ be their 
corresponding positions.
The highest possible position for $x^{(t)}_1$ is $a(t)$
and the lowest possible position for $x^{(t)}_{r(t)}$ is $b(t)$
where
\begin{align}
a(t)&= \begin{cases}
2(n-1)+t & t\leq q\\
2(n+q-1)-t & t\geq q
\end{cases}
&
b(t)&= \begin{cases}
-t & t\leq p\\
-2p+t& t\geq p.
\end{cases}
\end{align}
With $t$ in these equations regarded as a continuous parameter, $0 \le t \le p+q$, these lines
form the four non-vertical sides of the hexagon in Figure~\ref{fig:hexa}.

Let us introduce some fixed (virtual) particles by setting
\begin{align}
\label{2.1a}
x_{t+1}^{(t)} &:= -t - 2 = b(t) - 2 & (t&=0,\dots,p-1)
\intertext{and}
\label{2.1b}
x_0^{(t)} &:= 2(n+q)-t = a(t) + 2 & (t&=q,\dots,p+q-1).
\end{align}
Then it is a consequence of the combinatorics of this model, as
can easily be seen in Figure~\ref{fig:hexa}, that the blue particles
fulfill certain interlacing requirements,
\begin{equation}
x^{(t)}_{i+1} < x^{(t+1)}_{i+1} < x^{(t)}_{i},
\end{equation}
for $t=0$, \dots, $q-1$ and 
\begin{equation}
x^{(t)}_{i+1} < x^{(t+1)}_{i} < x^{(t)}_{i},
\end{equation}
for $t=q$, \dots, $p+q-1$. The probability of $\cup_{t=0}^{p+q} \{x^{(t)}_j\}_{j=1,\dots,r(t)} 
\}$ is therefore proportional to
\begin{equation}\label{3.2a}
\prod_{t=0}^{q-1} \tilde{\chi}(x^{(t)}, x^{(t+1)})
\prod_{t=q}^{p+q-1} \tilde{\chi}(x^{(t)}, x^{(t+1)}).
\end{equation}
Here $\tilde{\chi}(x,y)$ refers to the ordering between sets
$x=\{x_1,\dots,x_r\}$ and $y= \{y_1,\dots,y_r\}$ of the same
cardinality,
$$
x_r < y_r < \cdots < x_1 < y_1.
$$
Note that (\ref{3.2a}) is a discrete version of the joint PDF for
our finitization of the bead process, (\ref{3.2}).

The probability (\ref{3.2a}) can be written in terms of determinants.
Thus, using the identity
\begin{equation}\label{2.4a}
\chi_{x_r < y_{r-1} < \cdots < x_2 < y_1 < x_1} = \det [ \chi_{y_j > x_k} ]_{j,k=1,\dots,r}
\end{equation}
(see e.g.~\cite[Lemma 1]{FR02}), and writing $\phi(x,y) := \chi_{y>x}$, 
(\ref{3.2a}) reads 
\clabel{eqn:discmeasure}
\begin{equation}
\label{eqn:discmeasure}
\prod_{t=0}^{q-1}  \det [ \phi(x^{(t)}_i, x^{(t+1)}_j) ]_{i,j=1}^{r(t+1)}
\times
\prod_{t=q}^{p+q-1} \det [ \phi(x^{(t)}_i, x^{(t+1)}_{j-1}) ]_{i,j=1}^{r(t)} ,
\end{equation}
One way to think about this expression is that the blue particles
themselves are the positions of some dual random walks,
with step probability $\phi$. By the Lindstr\"om-Gessel-Viennot
method, the measure on configurations is
given exactly by (\ref{eqn:discmeasure}).

\begin{figure}
\includegraphics{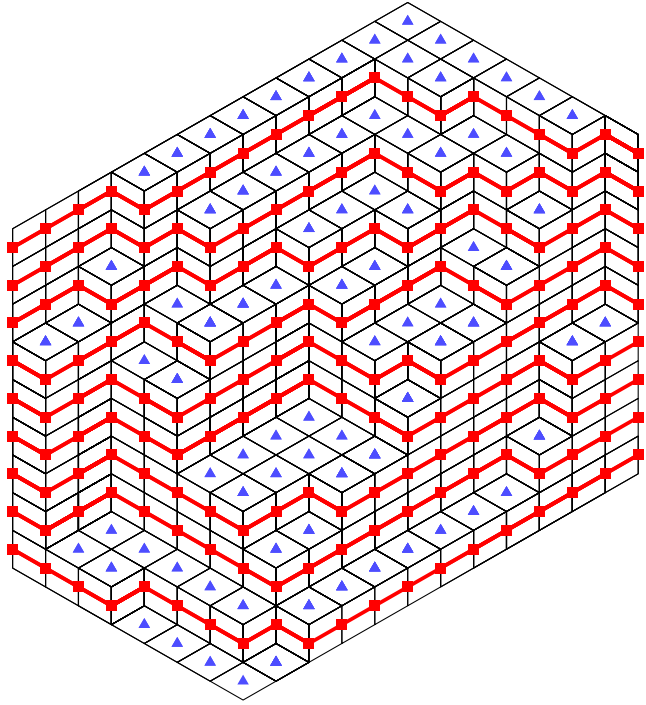}
\caption{In the text the spacing between lattice paths is 2 units, while in
the $x$-direction the spacing between lines is 1 unit, but in the figure these
lengths have been scaled so that the six internal corner angles of the hexagon are
equal. The three independent side lengths---left, bottom left, top left---are in
correspondence with the number of walkers $n$, and the parameters $p$ and $q$ of
the text respectively.
The red particles are marked by squares and the blue by triangles.}
\label{fig:hexa}
\end{figure}

A natural question to ask then is what is the probability
distribution of the blue particles at time $t$.
The idea is to compute the number of possible configurations
to the right and to the left of $t$ and divide by the total
number of possible tilings. For this we need the following
lemma.

\begin{lemma}
Let $t\leq p$.
Given some configuration $x^{(t)}$,
the number of configurations to the left of line $t$,
i.e.
\begin{equation}
H_t(x^{(t)})=\sum_{x^{(1)}, \dots, x^{(t-1)}}
\prod_{s =1}^{t-1}  \det [ \phi(x^{(s)}_i, x^{(s+1)}_j) ]_{i,j=1}^{s+1}
\end{equation}
is
\begin{equation}
H_t(x_1,\dots, x_t)= c_t
\Delta(x_1,\dots, x_t).
\end{equation}
where $\Delta$ is the Vandermonde determinant
and \begin{equation}
c_t=\frac{1}{2^{t(t-1)/2}\prod_{k =1}^{t-1} k! }.
\end{equation}
\end{lemma}
\begin{proof}
Induction over $t$.
For $t=1$ it is indeed true that $H_1(x)\equiv 1$.
Assuming the statement is true for $t$, consider
$t+1$. By summing the distinct variables along distinct rows
of the determinant we have
\begin{align}
H_{t+1}(x_1,\dots, x_{t+1}) &=
\sum_{x_i< y_i < x_{i+1}}
c_t \left| \begin{matrix}
1 & y_1 & \dots & y_1^{t-1}\\
\vdots\\
1 & y_t & \dots & y_t^{t-1}\\
\end{matrix}\right|
\\
&=c_t \left| \begin{matrix}
p_0(x_1)-p_0(x_2)  & p_1(x_1)-p_1(x_2) & \dots & p_{t-1}(x_1)-p_{t-1}(x_2)\\
\vdots\\
p_0(x_t)-p_0(x_{t+1})  & p_1(x_{t})-p_1(x_{t+1}) & \dots & p_{t-1}(x_{t})-p_{t-1}(x_{t+1})\\
\end{matrix}\right|
\intertext{where $p_k(x)=\sum_{y=2M}^{x/2} (2y)^k$ for $x$ even
and $p_k(x)=\sum_{y=2M}^{(x-1)/2} (2y+1)^k$ for $x$ odd, $M$ is an arbitrary
big negative number. 
It can be checked that $p_k$ is a polynomial and
that  $p_k(x)=x^{k+1}/2(k+1) + O(x^k)$. Performing column operations
in the determinant removes dependence on all but the highest order
coefficient in $p_k$ and gives}
&=c_t \left| \begin{matrix}
\frac{1}{2}(x_1-x_2)  & \frac{1}{4}(x_1^2-x_2^2) & \dots &\frac{1}{2t} (x_1^t-x_2^t)\\
\vdots\\
\frac{1}{2}(x_t-x_{t+1})  & \frac{1}{4}(x_t^2-x_{t+1}^2) & \dots &\frac{1}{2t} (x_t^t-x_{t+1}^t)\\
\end{matrix}\right|
\intertext{Pull out constants and border the matrix to make it $t+1$ by $t+1$ to obtain}
&=\frac{c_t}{2^t t!} \left| \begin{matrix}
0 & x_1-x_2  & x_1^2-x_2^2 & \dots &x_1^t-x_2^t\\
\vdots\\
0 & x_t-x_{t+1}  & x_t^2-x_{t+1}^2 & \dots &x_t^t-x_{t+1}^t\\
1 & x_{t+1} & x_{t+1}^2 & \dots &x_{t+1}^t
\end{matrix}\right|.
\end{align}
Row operations now complete the proof.
\end{proof}

\clabel{thm:discrete-model}
\begin{thm}
\label{thm:discrete-model}
The probability that at time $t$
the blue particles are at positions
$x_1> \dots > x_{r(t)}$ is
\begin{equation}
p_t(x_1, \dots, x_{r(t)} ) =
Z_t^{-1} \Delta^2(x_1, \dots, x_{r(t)}) \prod_{i=1}^{r(t)} f_t(x_i)
\end{equation}
where
\begin{equation}
f_t(x) = \prod_{k=1}^{|q-t|} (a(t)+2k-x)
\prod_{k=1}^{|p-t|} (x-b(t)+2k)
\end{equation}
and $Z_t$ is some normalizing constant.
\end{thm}
Note that this is called the Hahn ensemble in the literature, 
see for example~\cite{Jo02}.

\begin{proof}
\begin{figure}
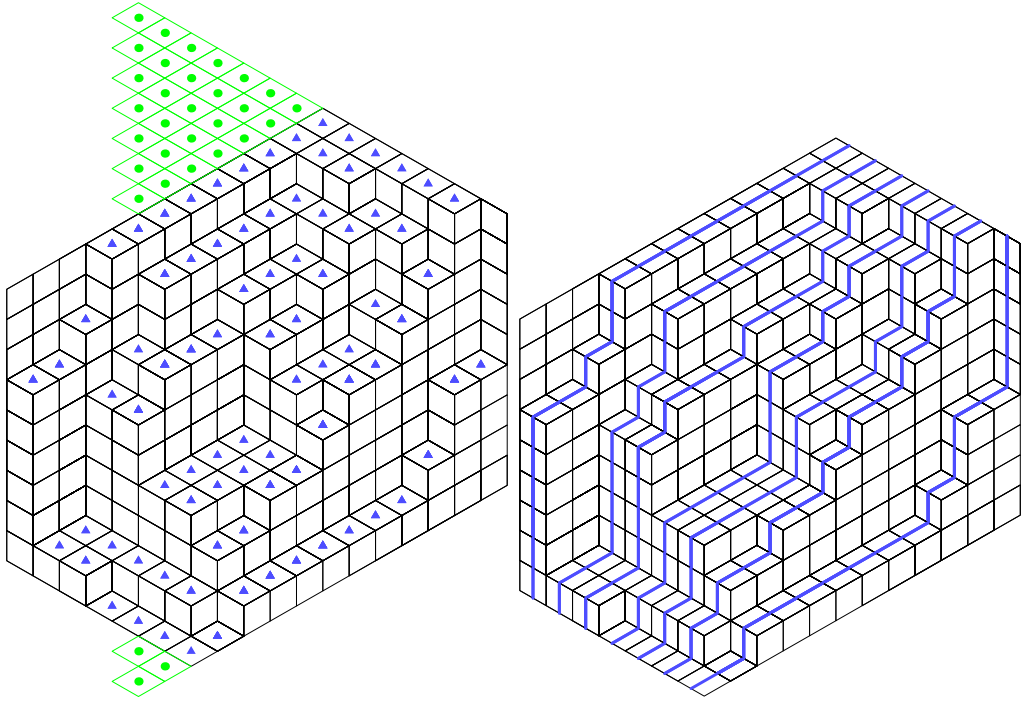

\includegraphics[scale=0.8]{fighexagonvirt.mps}
\includegraphics[scale=0.8]{fighexagonpaths.mps}
\caption{Construction of virtual particles as used in the case
$t \le p$ of the proof of Theorem~\ref{thm:discrete-model}.
The blue particles can be seen as 
the positions of some dual random walkers, 
leading to a probability distribution
proportional to the expression in~(\ref{eqn:discmeasure}).
The blue particles are marked with triangles while the 
virtual blue particles are green circles.}
\label{fig:hexb}
\end{figure}

Case $t\leq p$: The area to the left of $t$
can be tiled in $H_t(x_1,\dots, x_t)$ ways and the
area to the right of $t$ can be tiled in
$H_{p+q-t}(x_{-q+t+1}, \dots,  x_1,\dots, x_t, \dots, x_p)$ ways
where we introduce virtual particles
$x_{-q+t+i}=a(t)+2(q-t-i+1)$, for $i=1$, \dots, $q-t$
and $x_{t+i}= b(t)-2i$ for $i=1$, \dots, $p-t$
(see Figure~\ref{fig:hexb})
Then
\begin{equation}
p_t(x_1, \dots, x_{r(t)} ) =
\text{const}^{-1} H_t(x_1,\dots, x_t)H_{p+q-t}(x_{-q+t+1}, \dots,  x_1,\dots, x_t, \dots, x_p)
\end{equation}
which proves the theorem in this case since the
part of the Vandermonde determinant that has to do with the
virtual particles is exactly $f_t$.

Case $p\leq t\leq q$:
The area to the left of $t$
can be tiled in $H_t(x_1,\dots, x_p, \dots,  x_t)$ ways
where we introduce virtual particles $x_{p+i}=b(t)-2i$
for $i=1$, \dots, $t-p$. The area to the
right of $t$ can be tiled in exactly
$H_{p+q-t}(x_{-q+t+1},\dots,  x_1,\dots, x_p)$
ways where the virtual particles $x_{-q+t+i}=a(t)+2i$ for
$i=1$, \dots, $q-t$.
The theorem then follows just like in the first case.

Case $q\leq t$.
 The area to the left of $t$
can be tiled in $H_t(x_{q-t+1}, \dots, x_1,\dots, x_p)$ ways
with virtual particles $x_{q-t+i}= a(t)+q-t-2i$ for $i=1$, \dots, $t-q$
and $x_{p+q-t+i}=b(t)-2i$ for $i=1$, \dots, $t-p$.
Again the theorem follows.
\end{proof}

In~\cite{JN06} it is shown that the positions of 
all the blue particles---with the probability measure defined
by~(\ref{3.2a})---is a determinantal point process and 
its kernel is computed in terms of the Hahn polynomials. 
An alternative proof of Proposition~\ref{thm:jacobi_kernel}
would be to take their expression and 
applying the known limit formula 
\begin{equation}
\lim_{N\rightarrow\infty} Q_n(Nx, \alpha, \beta, N) = 
\frac{P_n^{\alpha,\beta}(1-2x)}{P_n^{\alpha,\beta}(1)}
\end{equation}
where $Q_n$ and $P^{\alpha,\beta}_n$ is the $n$:th Hahn polynomial 
respectively Jacobi polynomial. 
Our motivation for proving Proposition~\ref{thm:jacobi_kernel}
from first principles in section~\ref{sec:corr_functions}
is that our proof is more direct and intuitive. Also, the 
idea of virtual particles in fixed positions, as illustrated in 
the previous proof, has not been previously used to our knowledge.
That idea also gives a simple proof, see~\cite{No09}, of 
a result conjectured in~\cite{JN06} and proved in~\cite{OR06},
that the GUE minor process can be obtained as a scaling limit
of lozenge tilings of a hexagon close to the boundary. 

\section{Continuous model}

We are interested in the measure corresponding to the
probability~(\ref{eqn:discmeasure}) under the
rescaling $x^{(i)}_j \mapsto  x^{(i)}_j / n$ when $n\rightarrow\infty$
while $p$ and $q$ are kept fixed.
It is easy to see that this is exactly the measure corresponding to the
PDF~(\ref{3.2}). We have thus derived our finitized bead process as the
continuum limit of a tiling model.
In fact a semi-continuous lattice paths model (see Figure~\ref{fig:buses})
equivalent to our finitized bead process was
proposed in~\cite{BBDS06}
to model the bus transportation system in Cuernavaca (Mexico).
In that context, for $p \le t \le q$ the particle configuration $\{x_j^{(t)}\}_{j=1,\dots,r(t)}$
is the time at which bus number $j$ arrives at bus stop labelled $t$.

For future reference, we note that analogous to the equality between (\ref{3.2a}) and
(\ref{eqn:discmeasure}), use of~(\ref{2.4a}) allows~(\ref{3.2}) to be rewritten as
\begin{equation}
\label{2.5x}
\prod_{t=1}^{q-1}  \det [ \phi(x^{(t)}_i, x^{(t+1)}_j) ]_{i,j=1}^{r(t+1)}
\times
\prod_{t=q}^{p+q-1} \det [ \phi(x^{(t)}_i, x^{(t+1)}_{j-1}) ]_{i,j=1}^{r(t)} ,
\end{equation}
which apart from the absence of the $t=0$ term in the first product is
formally identical to~(\ref{eqn:discmeasure}). The variables take values on
different sets though, with those in~(\ref{eqn:discmeasure}) being confined to
lattice points, and furthermore in~(\ref{2.5x}) the virtual particles are now
specified by
\begin{align}\label{2.5x1}
x_{(t+1)}^{(t)} &= 0 \quad (t=1,\dots,p-1) &
x_0^{(t)} &= 1 \quad (t=q,\dots,p+q-1)
\end{align}
instead of (\ref{2.1a}), (\ref{2.1b}).

Rescaling the measure in Theorem~\ref{thm:discrete-model}
gives the following (cf.~\cite[eq.~(10)]{BBDS06}).

\begin{prop}
The projection of the PDF (\ref{3.2})
to $x^{(t)}$ is for
$x_1> \dots > x_{r(t)}$ the density
\begin{equation}\label{3.1x}
p_t(x_1, \dots, x_{r(t)} ) =
Z_t^{-1} \Delta^2(x_1, \dots, x_{r(t)}) \prod_{i=1}^{r(t)} f_t(x_i)
\end{equation}
where
\begin{equation}\label{ft}
f_t(x) = (1-x)^{|q-t|} x^{|p-t|}
\end{equation}
and $Z_t$ is some normalizing constant.
\end{prop}

\begin{figure}
\includegraphics[scale=0.7]{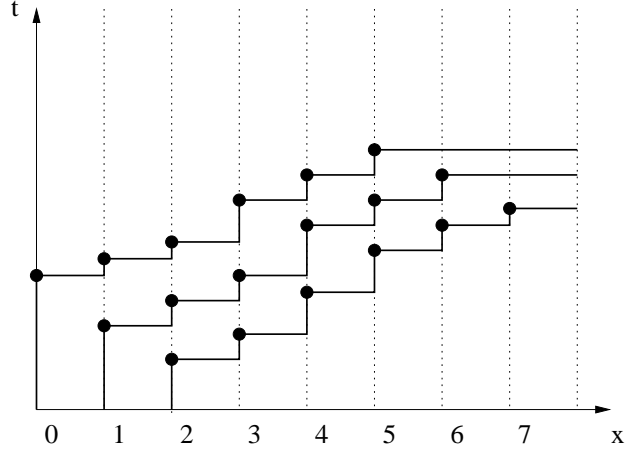}
\caption{In this lattice paths model introduced in \cite{BBDS06},
up/right corners correspond to arrival times at stations $2,3,4,5$.
Regarding all the up/right corners as particles our finitized bead process
results.}
\label{fig:buses}
\end{figure}

The PDF (\ref{3.1x}) is familiar in random matrix theory as the eigenvalue
probability density function for the Jacobi unitary ensemble (see e.g.~\cite{Fo02}).
Due to this interpretation, it has appeared in a number of previous studies, and some
properties of relevance to the bead process are known. Two of these are the support
of the limiting density, and the limiting density itself. To specify these results,
for notational convenience let us write $r(t)=n$ and 
\begin{equation}\label{ft1}
f_t(x) = x^\alpha (1 - x)^\beta,
\end{equation}
and consider (\ref{3.1x}) in the case that $\alpha = an$, $\beta = bn$ and $n \to \infty$.
Then we know from \cite{Co05,Fo02,Wa80} that the limiting support is the interval
$[c,d]$, where $c$ and $d$ are specified in terms of $a,b$ by
\begin{align}\label{be.1}
 c + d &= \frac{(2+a+b)^2 + (a^2 - b^2)}{ (2 + a + b)^2 } \nonumber \\
 d - c &= \sqrt{(c+d)^2 - 4cd},
\end{align}
and in this interval
\begin{equation}\label{be.2}
\tilde{\rho}_{(1)}(y) = \lim_{n \to \infty} \frac{1}{n} \rho_{(1)}(y) \Big 
|_{\begin{smallmatrix} \alpha = a n\\  \beta = bn \end{smallmatrix}} =
\frac{2 + a + b}{ 2 \pi} \frac{\sqrt{(y-c)(d-y)}}{y (1 - y) }.
\end{equation}

To make use of these results, let us introduce a scaled label $S$ for line $t$, and a
scaled measure of the difference in the left bottom and left top side measures
of the hexagon by
\begin{align}\label{Sp}
S &= t/p, & k &= (q-p)/p
\end{align}
respectively. We can then analyze the limit of the finitized bead process in which
$t,p,q \to \infty$ with $S,k$ fixed. Note that the scaled line label can take
on values
\begin{equation}
0 \le S \le 2 + k.
\end{equation}
Straightforward application of (\ref{be.1}) and (\ref{be.2}) reveals the following
limit theorem for the finitized bead process.

\begin{thm}\label{thm:b1}
In the above specified setting, the support of the bead process along scaled line
$S$ is $[c_S, d_S]$ with
\begin{align}\label{be.3}
& c_S =
 \frac{Sk}{(k+2)^2} + \frac{1}{k+2} - \frac{2\sqrt{S(k+1)(k+2-S)}}{(k+2)^2} \nonumber \\
& d_S =
 \frac{Sk}{(k+2)^2} + \frac{1}{k+2} + \frac{2\sqrt{S(k+1)(k+2-S)}}{(k+2)^2}.
\end{align}
The normalized density is given by (\ref{be.2}) with $c=c_S$, $d=d_S$ and $(a,b)$ given by
\begin{align}\label{cdab}
&\left ( \frac{1-S}{S}, \frac{k+1-S}{S} \right ), &
(S-1&, k+1-S), &
&\left ( \frac{S-1}{2+k-S}, \frac{S-k-1}{2+k-S} \right)
\end{align}
for $0\le S \le 1$, $1 \le S \le 1 +k$ and $1+k \le S \le 2+k$ respectively.
\end{thm}

We note that at $S=0$, $c_S = d_S = 1/(k+2)$, and similarly at $S=k+2$,
$c_S = d_S = (k+1)/(k+2)$. The degeneration of the support to a single point 
at the left and right sides is in keeping with the number of particles in the
bead process starting off at one at these sides. In (\ref{ft}), when $t=p$ the
exponent $\alpha$ in (\ref{ft1}) vanishes. Thinking of the factor $x^\alpha$ as
a repulsion from $x=0$, with this term not present we would expect that the
left support to be its smallest allowed value, $c_S=0$, which is indeed the
case in (\ref{be.3}) with $S=1$. Similarly, in  (\ref{ft}) with $t=q$ the
exponent $\beta$ in (\ref{ft1}) vanishes, and correspondingly in
 (\ref{be.3}) $d_S=1$ for $S=1+k$.

A particular realization of the boundary of support as implied by (\ref{be.3})
is plotted in Figure \ref{fig:bead.16}.

\section{Realization as a random matrix PDF}

We have seen that the joint PDF (\ref{3.2}) defining our finitization of the bead
process arises naturally in the context of the tiling of an $abc$-hexagon, or
equivalently by the consideration of non-intersecting paths. Here it will be
shown that (\ref{3.2}) can be obtained as the joint eigenvalue PDF of a sequence of
random matrices.

Our construction is based on theory related to certain random corank 1 projections contained
in \cite{Ba01a,FR02b}, which will now be revised. Let
\begin{align*} 
M &= \Pi A \Pi,&  \Pi &= {\mathbb I} - \vec{x} \vec{x}^\dagger
\end{align*}
where
\begin{equation*}
A = {\rm diag} \Big ( (a_1)^{s_1}, (a_2)^{s_2}, \dots , (a_n)^{s_n} \Big ).
\end{equation*}
Here the notation $(a)^p$ means $a$ is repeated $p$ times, and it is assumed
$a_1 > a_2 > \cdots > a_n$, while $\vec{x}$ is a normalized complex Gaussian vector
of the same number of rows as $A$. The eigenvalues $a_i$ of $A$ occur in $M$ with
multiplicity $s_i - 1$. Zero is also an eigenvalue of $M$. The remaining $n-1$
eigenvalues of $M$ occur at the zeros of the random rational function
\begin{equation}\label{rrf}
\sum_{i=1}^n \frac{q_i}{x - a_i}
\end{equation}
where $(q_1,\dots,q_n)$ has the Dirichlet distribution D$[s_1,\dots,s_n]$. With these
$n-1$ eigenvalues denoted $\lambda_1 > \cdots > \lambda_{n-1}$, it follows from this
latter fact that their joint distribution is equal to
\begin{equation}\label{2.1}
\frac{\Gamma(s_1 + \cdots + s_n) }{ \Gamma(s_1) \cdots \Gamma(s_n) }
\frac{\prod_{1 \le j < k \le n - 1} (\lambda_j - \lambda_k) }{
\prod_{1 \le j < k \le n} (a_j - a_k)^{s_j + s_k - 1} } 
\left(\prod_{j=1}^{n-1} \prod_{p=1}^n | \lambda_j - a_p |^{s_p - 1}\right)
\chi_{a_1 > \lambda_1 > a_2 > \cdots > \lambda_{n-1} > a_n}
\end{equation}
(as done in (\ref{1.2}), the indicator function will be abbreviated $\chi(\lambda, a)$).

After this revision, we begin the construction by forming $M_1 = \Pi A_1 \Pi$, where
$A_1 = {\rm diag} ( (0)^p, (1)^q)$. It follows from the above that $M_1$ has one
eigenvalue $\lambda_1^{(1)}$ different from 0 and 1 satisfying $0 < \lambda_1^{(1)} < 1$,
and this eigenvalue has PDF proportional to
\begin{equation*}
(\lambda_1^{(1)})^{p-1} (1 - \lambda_1^{(1)})^{q-1}.
\end{equation*}
Now, for $r=2,\dots,p$ inductively generate $\{\lambda_i^{(r)} \}_{i=1,\dots,r}$ as the 
eigenvalues different from 0 and 1 of the matrix
$$
M_r = \Pi A_r \Pi
$$
where
$$
A_r = {\rm diag} \Big ( (0)^{p-r+1}, \lambda_1^{(r-1)}, \dots, \lambda_{r-1}^{(r-1)},
(1)^{q-r+1} \Big ).
$$
It follows from (\ref{2.1}) that the PDF of $\{\lambda_j^{(r)} \}_{j=1,\dots,r}$, for given
$\{\lambda_j^{(r-1)} \}_{j=1,\dots,r-1}$, is proportional to
\begin{equation}\label{2.3}
\chi(\lambda^{(r-1)} \prec \lambda^{(r)} )
{ \prod_{i < j}^r (\lambda_i^{(r)} - \lambda_j^{(r)} )
\prod_{k=1}^r (1 - \lambda_k^{(r)})^{q-r} (\lambda_k^{(r)})^{p-r} \over
\prod_{i < j}^{r-1} (\lambda_i^{(r-1)} - \lambda_j^{(r-1)} )
\prod_{k=1}^{r-1} (1 - \lambda_k^{(r-1)})^{q-r+1} (\lambda_k^{(r-1)})^{p-r+1}}.
\end{equation}
Forming the product of (\ref{2.3}) over $r=1,\dots,p$ gives the joint PDF of
$ \cup_{s=1}^p \{ \lambda_i^{(s)} \}$ which is therefore proportional to
\begin{equation}\label{3.2b}
\prod_{r=2}^p \chi(\lambda^{(r-1)}, \lambda^{(r)} )
\prod_{i < j}^p (\lambda_i^{(p)} - \lambda_j^{(p)}) \prod_{k=1}^p (1 - \lambda_k^{(q)})^{q-p}.
\end{equation}

Next, for $r=1,\dots,q-p$, inductively generate $\{\lambda_i^{(p+r)} \}_{i=1,\dots,p}$
as the eigenvalues different from 0 and 1 of
$$
M_{p+r} = \Pi A_{p+r} \Pi
$$
where
$$
A_{p+r} = {\rm diag} \Big ( \lambda_1^{(p+r-1)},\dots , \lambda_p^{(p+r-1)}, (1)^{q-p-r+1} \Big ).
$$
We have from (\ref{2.1}) that the PDF of $\{\lambda_j^{(p+r)} \}_{j=1,\dots,p}$, for given
$\{ \lambda_j^{(p+r-1)} \}_{j=1,\dots,p}$, is proportional to
\begin{equation}\label{4.1}
\chi(\lambda^{(p+r-1)}, \lambda^{(p+r)} \cup \{ 0 \} )
{ \prod_{i < j}^p (\lambda_i^{(p+r)} - \lambda_j^{(p+r)} )
\prod_{k=1}^p  (1 - \lambda_k^{(p+r)})^{q-p-r} \over
\prod_{i < j}^{p} (\lambda_i^{(p+r-1)} - \lambda_j^{(p+r-1)} )
\prod_{k=1}^p (1 - \lambda_k^{(p+r-1)})^{q-r+1} }.
\end{equation}
The joint PDF of $ \cup_{s=1}^{q} \{ \lambda_j^{(s)} \}$ is obtained by multiplying
the product of (\ref{4.1}) over $r=1,\dots,q-p$ by (\ref{3.2b}). It is therefore
proportional to
\begin{equation}\label{4.2}
\prod_{r=2}^p \chi(\lambda^{(r-1)} , \lambda^{(r)} )
\prod_{r=p+1}^{q}  \chi(\lambda^{(r-1)} , \lambda^{(r)} \cup \{0\} )
\prod_{i < j}^p (\lambda_i^{(q)} - \lambda_j^{(q)}) .
\end{equation}

The final step is to inductively generate $\{\lambda_i^{(q+r)}\}_{i=1,\dots,p-r}$
$(r=1,\dots,p-1)$ as the eigenvalue different from 0 of
$$
M_{q+r} = \Pi A_{q+r} \Pi
$$
where
$$
A_{q+r} = {\rm diag} \Big ( \lambda_1^{(q+r-1)},\dots, \lambda_p^{(q+r-1)} \Big ).
$$
According to (\ref{2.1}), the PDF of $\{\lambda_j^{(q+r)}\}_{j=1,\dots,p-r}$ for given
$\{\lambda_j^{(q+r-1)}\}_{j=1,\dots,p-r+1}$ is proportional to
\begin{equation}
\chi(\lambda^{(q+r-1)} , \lambda^{(q+r)} \cup \{0,1\} )
{ \prod_{i < j}^{p-r} (\lambda_i^{(q+r)} - \lambda_j^{(q+r)} )
\prod_{k=1}^{p-r}  (1 - \lambda_k^{(q+r)})^{p-q-r} \over
\prod_{i < j}^{p-r-1} (\lambda_i^{(q+r-1)} - \lambda_j^{(q+r-1)} )
\prod_{k=1}^{p-r-1} (1 - \lambda_k^{(q+r-1)})^{q-r+1} }.
\end{equation}
Forming the product over $r=1,\dots,p-1$ and multiplying by the conditional PDF
(\ref{4.2}) gives that the joint PDF of $\cup_{s=1}^{q+p-1} \{ \lambda_j^{(s)} \}$
is proportional to
\clabel{eq:3}\begin{equation}
\label{eq:3}
\prod_{r=2}^p \chi(\lambda^{(r-1)} , \lambda^{(r)})
\prod_{r=p+1}^q \chi(\lambda^{(r-1)} , \lambda^{(r)} \cup \{0\})
\prod_{r=q+1}^{q+p-1} \chi(\lambda^{(r-1)}, \lambda^{(r)} \cup \{0,1\})
\end{equation}
and thus to the finite bead process (\ref{3.2}).

Graphical output from the implementation of this construction for particular
$p$ and $q$ is given in Figure \ref{fig:bead.16}.

\begin{figure}
\includegraphics{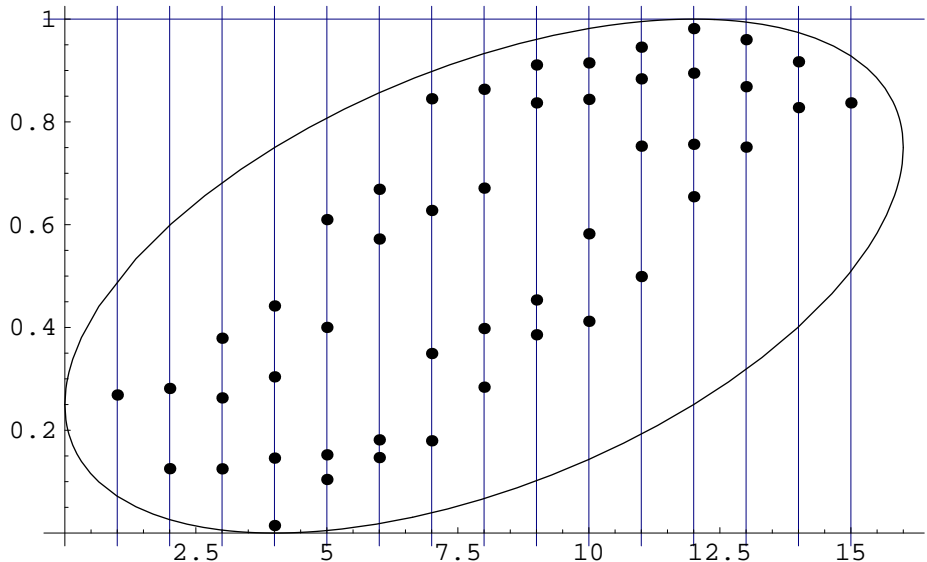}
\caption{A configuration of the finite bead process with
$p=4$, $q=12$, generated from the zeros of a sequence of random
rational functions (\ref{rrf}), corresponding to the eigenvalues of
the sequence of random matrices described in the text. The bounding curves are
the limiting shape as implied by (\ref{be.3}).}
\label{fig:bead.16}
\end{figure}

\section{Correlation functions}
\label{sec:corr_functions}
It was shown in \cite{JN06} that the GUE minor process as specified by the joint PDF
(\ref{1.2}) is determinantal, with the correlations between $n$ eigenvalues chosen
from any of the minors being given by an $n \times n$ determinant. The functional form
of the entries of the determinant is of particular interest. Thus we have from
 \cite{JN06} that
\begin{equation}
\rho_{(n)} (\{ (s_j,y_j) \}_{j=1, \dots, n} )=
\det[K^{\rm GUEm}((s_j, y_j), (s_k, y_k)) ]_{j,k=1,\dots, n}
\end{equation}
with
\clabel{eq:58}\begin{equation}
\label{eq:58}
K^{\rm GUEm}((s, x), (t, y)) =
\begin{cases}
\displaystyle
\frac{e^{-(x^2+y^2)/2} }{\sqrt{\pi}}
\sum_{k=1}^t
\frac{H_{s-k}(x)H_{t-k}(y)}{2^{t-k} (t-k)!},  &s\geq t\\ \\
\displaystyle
-\frac{e^{-(x^2+y^2)/2} }{\sqrt{\pi}}
\sum_{k=-\infty}^0
\frac{H_{s-k}(x)H_{t-k}(y)}{2^{t-k} (t-k)!},  &s\geq t
\end{cases}
\end{equation}
(here $H_j(x)$ denotes the Hermite polynomials of degree $j$). We seek an analogous formula for
the correlations of our finitized bead process. For this we make us of a linear algebra method
developed by Borodin and collaborators \cite{BR06,BFPS06}, although an alternative strategy would
be the method of Nagao and Forrester \cite{NF98a}. Either method can be used to calculate
(\ref{eq:58}), and various generalizations \cite{FN08,FN08a}. 

\subsection*{Formalism}
We take line $j$, which is the segment in the $xy$-plane $(j,t)$, $0 \le t \le 1$, and replace it
with a discretization ${\mathcal M}_j$. In the limit that $|{\mathcal M}_j| \to \infty$ it is required
that the discretization be dense on the segment. On the sets ${\mathcal M}_j \cup \{0\}$ $(j=1,\dots,p-1)$
distribute $j+1$ points $\{x_i^{(j)}\}_{i=1,\dots,j+1}$ with $x_{j+1}^{(j)} = 0$; on the sets
${\mathcal M}_{p-1+j}$ $(j=1,\dots,q-p+1)$
distribute $p$ points $\{x_i^{(p-1+j)}\}_{i=1,\dots,p}$; and on the sets
${\mathcal M}_{q+j} \cup \{1\}$ $(j=1,\dots,p-1)$ distribute $p+1-j$ points 
 $\{x_i^{(q+j)}\}_{i=0,\dots,p-j}$ with $x_{0}^{(q+j)} = 1$.

On the configuration of points ${\mathcal X} := \cup_{j=1}^{p+q-1}\{x_i^{(j)}\}$ define a (possibly signed)
measure
\begin{equation}
\label{2.5y}
{1 \over C} \prod_{t=1}^{q-1}  \det [ W_l(x^{(t)}_i, x^{(t+1)}_j) ]_{i,j=1}^{r(t+1)}
\times
\prod_{t=q}^{p+q-2} \det [ W_l(x^{(t)}_i, x^{(t+1)}_{j-1}) ]_{i,j=1}^{r(t)}
\end{equation}
for some functions $\{W_i : {\mathcal M}_i \times {\mathcal M}_{i+1} \rightarrow \mathbb R \}_{i=1,\dots,p+q-1}$.
With $W_l(y,x) = \phi(y,x) = \chi_{x > y}$, this corresponds to a discretization of the PDF
(\ref{2.5x}) for our finitized bead process.

Next introduce the $|{\mathcal M}_l|\times|{\mathcal M}_{l+1}|$ matrices 
$W_l = [W_l(x_i,y_j)]_{x_i \in {\mathcal M}_l, y_j \in {\mathcal M}_{l+1}}$.
Introduce too the $p \times |{\mathcal M}_{l+1}|$ matrices $E_l$ $(l=0, \dots, p-1)$ with entries in 
row $l+1$ all $1s$ and all other entries $0$, and the ${\mathcal M}_{p+q-l} \times p$ matrices 
$F_l$ $(l=1, \dots, p)$ with entries in column $l$ all $1s$ and all other entries $0$.
In terms of these matrices, with ${\mathcal M} := \{1, \dots, p\} 
\cup {\mathcal M}_1 \cup \dots \cup {\mathcal M}_{p+q-1}$, define the $|{\mathcal M}|\times|{\mathcal M}|$ matrix $L$ by
\[ \left[ \begin{array}{ccccccccccc}
0 & E_0 & E_1 & E_2 &\dots & E_{p-1} & 0 & \dots & \dots & \dots & 0 \\
0 & 0 & -W_1& 0 & \dots & 0 & 0 & \dots & \dots & \dots & 0\\
0 & 0 & 0 & -W_2 & \ddots & \vdots & \vdots & & & &\vdots\\
\vdots & \vdots & \vdots & \ddots & \ddots & 0 & \vdots & & & & \vdots \\
\vdots & \vdots & \vdots & \ddots & 0 & -W_{p-1} & 0& & & & \vdots \\
\vdots & 0 & \dots & \dots & \dots & 0 & -W_{p}& \ddots& & & \vdots \\
0 & \vdots & & & & &  \ddots & \ddots & 0  & & \vdots\\
F_{p-1} & \vdots & & & & & & 0 & -W_{q+1} & \ddots & \vdots\\
\vdots & \vdots& & & & & & & \ddots & \ddots & 0\\
 F_2 & 0 & \dots & \dots & \dots & \dots & \dots & \dots & \dots & 0 & -W_{p+q-2} \\
 F_1 & 0 & \dots & \dots & \dots & \dots & \dots & \dots & \dots & \dots &  0\\
\end{array} \right]\]
for block $0$ matrices of appropriate size. Note that the rows and columns of 
$L$ are labelled by the set ${\mathcal M}$. 

For $Y \subset {\mathcal M}$, 
introduce the notation $L_Y$ to denote the restriction of $L$ to the corresponding rows and columns.
Then, in accordance with the general theory of $L$-ensembles \cite{BR06}, our matrix $L$ has been constructed
so that the measure (\ref{2.5y}) is equal to 
\begin{equation}
\frac{\mbox{det } L_{\{1,\dots, p\}\cup {\mathcal X}}}{\det ({\bf 1}^{*} + L)}
\end {equation}
where ${\bf 1}^{*}$ is the $|{\mathcal M}|\times|{\mathcal M}|$ identity matrix with the 
first $p$ ones set to zero. The significance of this structure is the general fact 
that the correlation function for particles at $Y \subset  {\mathcal M} / \{1,\dots,p\}$ is then given by
\begin{align}
\rho (Y) &= \det K_Y,  &  K &= {\bf 1}^{*} - ({\bf 1}^{*} + L)^{-1} |_{{\mathcal M} / \{1,\dots,p\}}.
\end{align}

Our aim then is to evaluate $K$ in the continuum limit, 
with $W_i(y,x) = \chi_{y<x}$ independent of $i$, and to do this we must 
turn $({\bf 1}^{*} + L)^{-1}$ into a form that can be computed explicitly.
As a start, write $L$ in the structured form
\begin{equation}
L =  \begin{bmatrix}
0 & B \\
C & D-1\\
\end{bmatrix}
\end{equation}
where
\begin{align}
\label{eq:B}
B & = [E_0, E_1, \dots, E_{p-1}, 0, 0, \dots, 0]
\\ 
\label{eq:C} 
C & =  [0, 0, \dots, 0, F_{p-1}, F_{p-2}, \dots, F_1]^T 
\end{align}
We know that \cite{BR06}
\begin{equation}
K = {\bf 1}^{*}  - D^{-1} + D^{-1}CM^{-1}BD^{-1}
\end{equation}
with $M := BD^{-1}C$ and furthermore
\begin{equation}
\label{eq:D} D^{-1} = {\bf 1} + [W_{[i,j)}]_{i,j=1,\dots,p+q-1}
\end{equation}
for
\begin{equation} \label{eq:W}
W_{[i,j)} = \begin{cases}
W_i \dots W_{j-1}, & i  <  j\\
0, & i  \geq  j.
\end{cases}
\end{equation}

From (\ref{eq:D}) we compute that the $m$-th member of the block row vector $BD^{-1}$ is equal to
\begin{equation} \label{eq:BDm} E_{m-1} + \sum_{k=1}^{m-1} E_{k-1}W_{[k,m)}  \end{equation}
for $1\leq m \leq p$ and
\begin{equation} \label{eq:BDm2} \sum_{k=1}^{p} E_{k-1} W_{[k,m)}  \end{equation}
for $p+1 \leq m \leq p+q-1$.
Similarly the $m$-th member of the block column vector $D^{-1}C$ is equal to
\begin{equation}\label{eq:DC} \sum_{k=1}^{p-1} W_{[m, q+k)} F_{p-k} \end{equation}
for $ 1 \leq m \leq q$ and
\begin{equation}\label{eq:DC2} F_{p+q-m} + \sum_{k=1}^{p+q-1-m} W_{[m, m+k)} F_{p+q-m-k} \end{equation}
for $q+1 \leq m \leq p+q-1$.

\subsection*{Explicit formulas}
Our aim is to find square, $p \times p$ matrices $B_0$, $C_0$ such 
that: $BD^{-1} = B_0\Phi$ for some block row vector $\Phi = [\Phi^1,\dots, 
\Phi^{p+q-1}]$; $D^{-1}C = \Psi C_0$ for some block column vector 
$\Psi = [\Psi^1,\dots, \Psi^{p+q-1}]$; and $B_0C_0=M$.  
Finding these square matrices will allow us to compute $K$ explicitly, 
since if the above equalities hold, then 
\begin{equation}\label{K1} K = {\bf 1}^{*}  - D^{-1} + \Psi\Phi \end{equation}
We introduce the notation
\begin{equation} (a*b)(y,x) = \int_0^1 a(y,z)b(z,x)dz \end{equation}
and note that, in the continuum limit 
${\mathcal M}_{i} \rightarrow [0,1]$, $(i=1, \dots, p+q-1)$, and with the
lattice points weighted by their mean spacing, (\ref{eq:W}) becomes
\begin{equation}
W_{[i,j)}(y,x) =  \begin{cases}
(W_i * \dots * W_{j-1})(y,x), & i  <  j\\
0, & i  \geq  j\\
\end{cases}
\end{equation}
Independent of $i$, $W_i(y,x) = \chi_{y<x}$ and so
\begin{equation}\label{eq:Wyx}
W_{[i,j)}(y,x) = \frac{1}{(j-i-1)!}\chi_{y<x}(x-y)^{j-i-1}
\end{equation}
With the entry of $K$ corresponding to line $s$, position $y$ for the row, and line
$t$, position $x$ for the column, denoted $K(s,y;t,x)$ it follows from
(\ref{eq:D}) , (\ref{K1}) and (\ref{eq:Wyx}) that
\begin{equation}\label{K}
K(s,y;t,x) = - \frac{1}{(t-s-1)!} (x-y)^{t-s-1} \chi_{y<x} + \sum_{l=1}^{p} (\Psi^s)_{y,l} (\Phi^t)_{l,x}
\end{equation}
so our task is to calculate the quantities $(\psi^s)_{y,l}$, $(\Phi^t)_{l,x}$.

We begin by looking at the $m$-th member of the block row vector 
$BD^{-1}$ for $m\leq p$ as shown in (\ref{eq:BDm}).  Using (\ref{eq:Wyx}) we see that
\begin{equation}\label{eq:xmp}
\left(E_{m-1} + \sum_{k=1}^{m-1} E_{k-1}W_{[k,m)}\right)_{l,x} = \frac{x^{m-l}}{(m-l)!}
\end{equation}
where the convention $\frac{1}{a!} =0$ for $a \in \mathbb Z_{<0}$ is to be used for $m<l$.
Next define the family of polynomials
\begin{equation} \label{poly}
\tilde{P}_n^{(a,b)}(x) =  P_n^{(a,b)}(1-2x) = \left\{\begin{array}{lc}
\displaystyle \frac{1}{n!}\frac{1}{x^a(1-x)^b}\frac{d^n}{dx^n} \left(x^{n+a}(1-x)^{n+b}\right), & n\geq 0 \\
\displaystyle 0\, , & n < 0 \end{array} \right.
\end{equation}
These polynomials have orthogonality
\begin{equation}\label{orthog}
\int_0^1 x^a(1-x)^b \tilde{P}_j^{(a,b)}(x) \tilde{P}_k^{(a,b)}(x) dx = {\mathcal N}_j^{(a,b)} \delta_{jk}
\end{equation}
where
\begin{equation}
{\mathcal N}_n^{(a,b)} = \frac{1}{2n+a+b+1}\frac{(n+a)!(n+b)!}{n!(n+a+b)!}
\end{equation}
and as such, as indicated in (\ref{poly}),
are the Jacobi polynomials $P_n^{(a,b)}$ modified to be orthogonal on $(0,1)$.

Expressing (\ref{eq:xmp}) as a sum of these polynomials for certain $a=p-m, \,b=q-m$ gives
\begin{eqnarray}\label{5.24}
\frac{x^{m-l}}{(m-l)!} & = & \frac{1}{(m-l)!}\sum_{j=l}^m 
\frac{\tilde{P}_{m-j}^{(p-m,\,q-m)}(x)}{{\mathcal N}_{m-j}^{(p-m, \,q-m)}} 
\int_0^1 t^{m-l}\,t^{p-m} (1-t)^{q-m} \tilde{P}_{m-j}^{(p-m,\,q-m)}(t) dt \nonumber \\
&  = & \sum_{j=1}^m \frac{(-1)^{m+j}}{(m-j)!}\frac{\tilde{P}_{m-j}^{(p-m,\,q-m)}(x)}{{\mathcal N}_{m-j}^{(p-m, \,q-m)}} 
\frac{1}{(j-l)!} \int_0^1t^{p-l} (1-t)^{q-l} dt,
\end{eqnarray}
where the second equality follows using (\ref{poly}) and integrating by parts $m-j$ times.
If we let
\begin{equation}
(B_0)_{l,n} = \frac{(-1)^{p+n}}{{\mathcal N}_{p-n}^{(0,q-p)}(p-n)! (n-l)!} \int_0^1t^{p-l} (1-t)^{q-l} dt
\end{equation}
and
\begin{equation}\label{phi}
(\Phi^m)_{n,x}= \left\{ \begin{array}{lc} \displaystyle (-1)^{p+m}\frac{(p+q-m-n)!}{(q-n)!}\tilde{P}_{m-n}^{(p-m,\, q-m)}(x) & n \leq m\\
\displaystyle 0 & n > m
\end{array} \right.
\end{equation}
then
\begin{equation} \label {eq:m-p} \frac{x^{m-l}}{(m-l)!} = \sum_{n=1}^{p} (B_0)_{l,n}(\Phi^m)_{n,x} \end{equation}
and therefore
\begin{equation} \label{eq:rel}
E_{m-1} + \sum_{k=1}^{m-1} E_{k-1}W_{[k,m)} = B_0\Phi^m
\end{equation}

Consider next the $m=(p+j)$-th member of the block row vector $BD^{-1}$ as shown in (\ref{eq:BDm2}). Using (\ref{eq:Wyx}) we see that
\begin{equation}\label{5.29}
\left(\sum_{k=1}^{p} E_{k-1} W_{[k,p+j)}\right)_{l,x} = \frac{x^{p+j-l}}{(p+j-l)!}
\end{equation}
To start with, we look at the $m=p$ case of (\ref{eq:m-p})
\begin{equation}\label{eq:N-p}
\frac{x^{p-l}}{(p-l)!} = \sum_{n=1}^{p} (B_0)_{l,n} \, \tilde{P}^{(0,q-p)}_{p-n}(x)
\end{equation}
We introduce the operator $J[f(x)] = \int_0^x f(t)dt$ and note that
\begin{equation}
J^j[f(x)] = \frac{1}{(j-1)!} \int_0^x (x-t)^{j-1} f(t) dt
\end{equation}
Applying $J^j$ to both sides of (\ref{eq:N-p}) gives
\begin{equation}\label{eq:Jj}
\frac{x^{p+j-l}}{(p+j-l)!} = \sum_{n=1}^{p} \frac{(B_0)_{l,n}}{(j-1)!} \int_0^x(x-t)^{j-1} \, 
\tilde{P}^{(0,q-p)}_{p-n}(t) dt
\end{equation}
If we let
\begin{equation}\label{phi2}
(\Phi^{p+j})_{n,x} = \frac{1}{(j-1)!} \int_0^x(x-t)^{j-1} \, \tilde{P}^{(0,q-p)}_{p-n}(t) dt
\end{equation}
then it follows from (\ref{5.29}) and (\ref{phi2}) that
\begin{equation}
\sum_{k=1}^{p} E_{k-1} W_{[k,p+j)} = B_0\Phi^{p+j}
\end{equation}
and this coupled with (\ref{eq:rel}) gives us our desired $BD^{-1} = B_0\Phi$.

Now that we have defined $B_0$, we can set about finding what $C_0$ must be, using the equation
\begin{equation}\label{BCM}
B_0 C_0 = M
\end{equation}
 First,
using (\ref{eq:B}), (\ref{eq:C}) and (\ref{eq:D}) we have
\begin{equation}
M = \left[\sum_{j=1}^{p-1}\left(\sum_{k=1}^{p}E_{k-1} W_{[k,q+j)}\right)F_{p-j}\right].
\end{equation}
The formula (\ref{5.29}) for the inner sum implies
\begin{equation}\label{Mp}
(M)_{j,k} = \frac{1}{(p+q+1-j-k)!}
\end{equation}

To proceed further requires the first of
Jacobi polynomial identities 
\begin{eqnarray}
\label{dx} \frac{d}{dx} \left( x^a \tilde{P}_n^{(a,b)}(x)\right) & = & (n+a)x^{a-1}P_n^{(a-1,b+1)}(x) \\
\label{dx2} \frac{d}{dx} \left( (1-x)^b \tilde{P}_n^{(a,b)}(x)\right) & = & -(n+b)(1-x)^{b-1}P_n^{(a+1,b-1)}(x),
\end{eqnarray}
which can be verified from (\ref{poly}) (the second will be of use subsequently).
Setting $j = q+1-m, l=n$ and $x=1$ in (\ref{eq:Jj}) gives
\begin{equation}\label{5.39}
\frac{1}{(p+q+1-n-m)!}  =  
\sum_{k=1}^{p} (B_0)_{n,k} \int_0^1\frac{(1-t)^{q-m}}{(q-m)!} \, \tilde{P}^{(0,q-p)}_{p-k}(t) dt 
\end{equation}
But it follows from (\ref{dx}) that
\begin{equation}
\tilde{P}^{(0,a)}_n(x) = \frac{n!}{(n+a)!} \frac{d^a}{dx^a}(x^a
\tilde{P}_n^{(a,0)}(x)), \qquad a \in \mathbb Z_{\ge 0},
\end{equation}
so the integral in (\ref{5.39}) can be reduced by integration by parts, giving 
\begin{equation}
\frac{1}{(p+q+1-n-m)!}     
 =  \sum_{k=1}^{p} \frac{(B_0)_{n,k}}{(k-n)!} \int_0^1 \frac{t^{q-k}(1-t)^{p-m}}{(q-k)!} dt
\end{equation}
Recalling (\ref{BCM}) and (\ref{Mp}) we thus have
\begin{equation}
(C_0)_{k,m} = \frac{1}{(k-m)!(q-k)!} \int_0^1 t^{q-k}(1-t)^{p-m} dt\
\end{equation}

We now look at the $m=(q+j)$-th member of the block column vector $D^{-1}C$ as shown in (\ref{eq:DC2}).  We see that
\begin{equation} \label{yjp}
\left(F_{p+1-j} + \sum_{k=1}^{p-1-j} W_{[q+j, q+j+k)} F_{p-j-k}\right)_{y,n} = \frac{(1-y)^{p+1-j-n}}{(p+1+j-n)!}
\end{equation}
Expressing this as a sum of the polynomials from (\ref{poly}) for certain $a=j+q-p,b=j$ gives
\begin{eqnarray} \label{yjp1}
\frac{(1-y)^{p-j-n}}{(p-j-n)!} & = & \frac{1}{(p-j-n)!}\sum_{k=n}^{p-j} 
\frac{\tilde{P}^{(j+q-p, j)}_{p-j-k}(y)}
{{\mathcal N}^{(j+q-p, j)}_{p-j-k}}\nonumber \\
& & \times \int_0^1 (1-t)^{p-j-n}\, t^{j+q-p} (1-t)^{j} \tilde{P}^{(j+q-p, j)}_{p-j-k}(t) dt \nonumber \\
\label{Njq}& = & \sum_{k=1}^{p-j} (C_0)_{k,n} \frac{(q-k)!}{(p-j-k)!}\frac{\tilde{P}^{(j+q-p, j)}_{p-j-k}(y)}
{{\mathcal N}^{(j+q-p, j)}_{p-j-k}}
\end{eqnarray}
We remark that this equation can alternatively be derived from (\ref{5.24}) by writing $x=1-y$, $m=p+1-j$, $l=n$,
and noting the symmetry of the Jacobi polynomials $\tilde{P}_n^{(a,b)}(1-x) = (-1)^n \tilde{P}_n^{(b,a)}(x)$.
It follows from (\ref{yjp}) and (\ref{yjp1}) that by setting
\begin{equation}\label {psi}
(\Psi^{q+j})_{y,k} = \left\{ \begin{array}{lr} \displaystyle \frac{(q-k)!}{(p-j-k)!}
\frac{\tilde{P}^{(j+q-p, j)}_{p-j-k}(x)}
{{\mathcal N}^{(j+p-q, j)}_{p-j-k}} & k \leq p-j \\
 \displaystyle 0 & k> p-j \end{array} \right.
\end{equation}
we obtain
\begin{equation}\label{rel}
F_{p-j} + \sum_{k=1}^{p-1-j} W_{[q+j, q+j+k)} F_{p-j-k} = \Psi^{q+j} C_0
\end{equation}

Finally, we look at the $m$-th block of the column vector $D^{-1}C$ as shown in (\ref{eq:DC}). We see that
\begin{equation}
\left(\sum_{k=1}^{p-1} W_{[m, q+k)} F_{p-k}\right)_{y,n} = \frac{(1-y)^{p+q-m-n}}{(p+q-m-n)!}
\end{equation}
With $m=q$ this corresponds to  
$ j=0$ case in (\ref{Njq}), so we have
\begin{equation}
\frac{(1-y)^{p-n}}{(p-n)!} = \sum_{k=1}^{p} (C_0)_{k,n} \frac{(q-k)!}{(p-k)!}
\frac{\tilde{P}^{(q-p, 0)}_{p-k}(y)}{{\mathcal N}^{(q-p, 0)}_{p-k}}
\end{equation}
After making a substitution $1-y = w$ and applying $J^j$, $j=q-m$,  we end up with
\begin{eqnarray}
\frac{(1-y)^{p+q-m-n}}{(p+q-m-n)!} & = & \sum_{l=1}^{p} \frac{(C_0)_{l,n}}{{\mathcal N}^{(q-p, 0)}_{p-l}} 
\frac{(q-l)!}{(p-l)!}
\int_y^1\frac{(z-y)^{q-1-m}}{(q-1-m)!} \tilde{P}_{p-l}^{(q-p,0)}(z) dz
\end{eqnarray}
Setting
\begin{equation}\label{psi2}
(\Psi^m)_{y,l} = \frac{1}{{\mathcal N}^{(q-p, 0)}_{p-l}}\frac{(q-l)!}{(p-l)!} 
\int_y^1\frac{(z-y)^{q-m}}{(q-m)!} \tilde{P}_{p-l}^{(q-p,0)}(z) dz
\end{equation}
gives us
\begin{equation}
\sum_{k=1}^{p-1} W_{[m, q+k)} F_{p-k} = \Psi^mC_0
\end{equation}
and this, along with (\ref{rel}) gives us our desired $D^{-1}C = \Psi C_0$.

The quantities  $(\Psi^s)_{y,l}$ and $(\Phi^t)_{l,x}$ in (\ref{K}) have now been specified in
(\ref{phi}), (\ref{phi2}), (\ref{psi}) and (\ref{psi2}), allowing us (after a little bit more
calculation)  to write down the
explicit form of the entries for the correlation function in terms of Jacobi polynomials.

\begin{prop}
\label{thm:jacobi_kernel}
Let $(t,x)$ refer to the position $x$ on line $t$ of the bead process.  
For the configuration $Y = \bigcup_{i=1}^n \{(t_i,x_i)\}$ the correlation function is given by
\begin{equation}\label{rK}
\rho(Y) = \det [K(t_i, x_i; t_j, x_j)]_{i,j=1,\dots, n}
\end{equation}
where
\begin{equation}\label{final}
K(s,y;t,x) = \left\{ \begin{array}{rc} \displaystyle a_s(y)b_t(x) \sum_{l=1}^{\alpha(s,t)} 
\frac{{\mathcal C}_l^{(s)}}{{\mathcal C}_l^{(t)}}\frac{Q_l^{(s)}(y) Q_l^{(t)}(x)}{{\mathcal N}_l^{(t)}} , & s \geq t\\
\displaystyle -a_s(y)b_t(x) \sum_{l=-\infty}^{0} \frac{{\mathcal C}_l^{(s)}}{{\mathcal C}_l^{(t)}}\frac{Q_l^{(s)}(y) 
Q_l^{(t)}(x)}{{\mathcal N}_l^{(t)}},  & s< t \end{array} \right.
\end{equation}
Here
\begin{equation}
a_s(y)  =  \left\{\begin{array}{lrcccl}
(-y)^{p-s} (1-y)^{q-s} &  0  <  s  \leq  p\\
(1-y)^{q-s} & p  < s \leq  q\\
1 & q  <  s  \leq  p+q \end{array}\right.
\end{equation}
\begin{equation}
b_t(x)  =  \left\{\begin{array}{lrcccl}
(-1)^{p-t} & 0  <  t  \leq  p\\
x^{t-p} & p  <  t \leq  q\\
x^{t-p} (1-x)^{t-q} &  q  <  t  \leq  p+q-1 \end{array}\right.
\end{equation}
\begin{equation}
{\mathcal C}_l^{(t)}  =  \left\{\begin{array}{lrcccl}
\displaystyle \frac{(t-l)!}{(p-l)!} & 0  <  t  \leq  p\\
\displaystyle \frac{(q-l)!}{(p+q-t-l)!}  & p  <  t \leq  q\\
\displaystyle \frac{(t-l)!}{(p-l)!} & q  <  t  \leq  p+q-1 \end{array}\right.
\end{equation}
\begin{equation}\label{5.59}
Q_l^{(t)}(x)  =  \left\{\begin{array}{lrcccl}
\tilde{P}_{t-l}^{(p-t, q-t)}(x) & 0  <  t  \leq p\\
\tilde{P}_{p-l}^{(t-p, q-t)}(x) & p  <  t \leq q\\
\tilde{P}_{p+q-t-l}^{(t-p, t-q)}(x) & q  <  t  \leq  p+q-1 \end{array}\right.
\end{equation}
\begin{equation}
\mathcal{N}_l^{(t)}(x)  =  \left\{\begin{array}{lrcccl}
\mathcal{N}_{t-l}^{(p-t, q-t)} & 0  <  t  \leq  p\\
\mathcal{N}_{p-l}^{(t-p, q-t)} & p  <  t \leq  q\\
\mathcal{N}_{p+q-t-l}^{(t-p, t-q)} & q  <  t  \leq  p+q-1 \end{array}\right.
\end{equation}
and $\alpha(s,t) = \min [p+q-s, t, p]$.
\end{prop}
\begin{proof}
The $s \geq t$ case follows directly from (\ref{K}) by inputing the values given in (\ref{phi}), (\ref{phi2}), (\ref{psi}) and (\ref{psi2}), and using (\ref{poly}), (\ref{dx}) and (\ref{dx2}).
For $s<t$ some further calculation is required, as in this case the first term in (\ref{K}) is non-zero.  In this case we express this term as a sum of polynomials in $x$
\begin{equation}
\frac{1}{(t-s-1)!} (x-y)^{t-s-1} \chi_{y<x} = b_t(x)\sum_{l=-\infty}^{\infty} \alpha_l Q_l^{(t)}(x)
\end{equation}
(or similarly in $y$), with the aim to obtain the form in (\ref{final}).  It turns out that this expansion is such that 
\begin{eqnarray}
b_t(x)\sum_{l=-\infty}^{0} \alpha_l Q_l^{(t)}(x) & = & a_s(y)b_t(x) \sum_{l=-\infty}^{0} 
\frac{{\mathcal C}_l^{(s)}}{{\mathcal C}_l^{(t)}}\frac{Q_l^{(s)}(y)Q_l^{(t)}(x)}{{\mathcal N}_l^{(t)}} \\
b_t(x)\sum_{l=1}^{\infty} \alpha_l Q_l^{(t)}(x) & = & \sum_{l=1}^{p} (\Psi^s)_{y,l} (\Phi^t)_{l,x}
\end{eqnarray}

For example, for $p \leq s <  t \leq q$
\begin{equation}
\sum_{l=1}^{p} (\Psi^s)_{y,l} (\Phi^t)_{l,x}  = x^{t-p}(1-y)^{q-s} \sum_{l=1}^{p} \frac{(s-l)!}{(t-l)!}
\frac{\tilde{P}_{p-l}^{(t-p, q-t)}(x) 
\tilde{P}_{p-l}^{(s-p, q-s)}(y)}{{\mathcal N}_{p-l}^{(s-p, q-s)}}
\end{equation}
We expand in $y$
\begin{equation}
\frac{1}{(t-s-1)!} (x-y)^{t-s-1} \chi_{y<x} = (1-y)^{q-s}\sum_{l=0}^\infty \alpha_l \tilde{P}_{l}^{(s-p, q-s)}(y)
\end{equation}
Using the orthogonality (\ref{orthog})
\begin{equation}
{\mathcal N}_{l}^{(s-p, q-s)}\alpha_l = \int_0^x \frac{(x-u)^{t-s-1}}{(t-s-1)!} u^{s-p} 
\tilde{P}_{l}^{(s-p, q-s)}(u) du
\end{equation}
Using (\ref{dx}) and integration by parts $(t-s)$ times gives
\begin{equation}
{\mathcal N}_{l}^{(s-p, q-s)}\alpha_l = \frac{(l+s-p)!}{(l+t-p)!} x^{t-p}\tilde{P}_l^{(t-p, q-t)}(x)
\end{equation}
so
\begin{align}
\frac{1}{(t-s-1)!} (x-y)^{t-s-1} \chi_{y<x} & =  x^{t-p}(1-y)^{q-s}\sum_{l=-\infty}^{p} \frac{(s-l)!}{(t-l)!} 
\frac{\tilde{P}_{p-l}^{(t-p, q-t)}(x)\tilde{P}_{p-l}^{(s-p, q-s)}(y)}{{\mathcal N}_{l}^{(s-p, q-s)}}  \\
&= \sum_{l=1}^{p} (\Psi^s)_{y,l} (\Phi^t)_{l,x} + a_s(y)b_t(x) \sum_{l=-\infty}^{0} 
\frac{{\mathcal C}_l^{(s)}}{{\mathcal C}_l^{(t)}}\frac{Q_l^{(s)}(y)Q_l^{(t)}(x)}{{\mathcal N}_l^{(t)}}
\end{align}
as required.
\end{proof}

\section{Bulk scaling}
In Section 3, a scaled limit of the bead process was considered in which the lines themselves formed
a continuum. We computed the support of the density, and the density profile. In this section we
focus attention on a different scaled limit---referred to as bulk scaling---in which the particles
are on lines a finite distance (as measured by the line number difference) apart, with the interparticle
spacing on the lines scaled to be unity. We require too that the particles be away from the boundary of
support (thus the use of the term bulk scaling), which we do by requiring them to be in the neighbourhood
of the midpoint of the support, $(c_S + d_S)/2$ in the notation of Theorem \ref{thm:b1}.

Explicitly, the location $t$ of the lines will be measured from a reference line of continuum label $S$
(recall (\ref{Sp}), and thus $t \mapsto p S + t$. With $X_S := (c_S + d_S)/2$ denoting the midpoint of the
support on line $S$, and 
\begin{equation}
u_S := \tilde{\rho}_{(1)}(X_S)
\end{equation}
where $ \tilde{\rho}_{(1)}$ is the normalized density
(\ref{be.2}) on line $S$ the location $x$ of the particles are to be chosen so that
$x = X + X_i/(r(pS) u_S)$. (The quantity $r(t)$ denotes the number of particles on line $t$ as specified
by (\ref{3.1})). Our objective in this section is to compute the scaled correlation kernel
\begin{equation}\label{cK}
\bar{K}(s_0,Y;t_0,X) := \lim_{p \to \infty}
\frac{1}{r(pS) u_S} K \Big ( pS + s_0, y;  pS + t_0, x \Big )
\end{equation}
with
\begin{align}\label{6.1a}
x &= X_S + {X \over r(pS) u_S}, &
y &= X_S + {Y \over r(pS) u_S},
\end{align}
and thus according to (\ref{rK}) the  scaled correlation
function
\begin{multline}\label{6.1b}
 \bar{\rho}(t_1,X_1;\dots;t_n,X_n) \\
 = 
\lim_{p \to \infty} \Big ( \frac{1}{r(pS) u_S} \Big )^n
\rho (pS + t_1, X_S + X_1/(r(pS) u_S); \dots; pS + t_n, X_S + X_n/(r(pS) u_S) ).
\end{multline}

The strategy to be adopted is to make use of a known asymptotic expansion for the large $n$
form of $P^{\alpha + an, \beta + bn}_n(x)$ with $x$ such that the leading behaviour is oscillatory
(outside this interval there is exponential decay) \cite{CI91}. Based on experience with similar 
calculations \cite{FNH99, FN08, FN08a} we expect that this will show that to leading order
the sums are Riemann sums, and so turn into integrals in the limit $p \to \infty$.

\begin{prop}\label{chi}
Let $a,b,x$ be given, and define the parameters $\Delta, \rho, \theta, \gamma$ according to
\begin{align}\label{18}
& \Delta = [a(x+1)+b(x-1)]^2-4(a+b+1)(1-x^2) & \nonumber \\
\frac{2e^{i \rho}}{\sqrt{(1+a+b)(1-x^2)}}  & =   
\frac{a(x+1) + b(x-1) + i\sqrt{-\Delta}}{(1+a+b)(1-x^2)}   & -\pi < \rho \leq \pi, \nonumber \\
\sqrt{\frac{2(a+1)}{(1-x)(1+a+b)}}e^{i \theta}   & =   \frac{(a+b+2)x-(3a+b+2)-i\sqrt{-\Delta}}{2(x-1)(1+a+b)} 
& -\pi < \theta \leq \pi \nonumber \\
\sqrt{\frac{2(b+1)}{(1+x)(1+a+b)}}e^{i \gamma} & =   
\frac{(a+b+2)x+(a+3b+2)-i\sqrt{-\Delta}}{2(x+1)(1+a+b)}& -\pi < \gamma \leq \pi
\end{align}
For $\Delta < 0$ we have the large $n$ asymptotic expansion
\begin{multline}\label{21}
P_n^{(\alpha+an, \beta + bn)}(x) \sim  \left(\frac{4}{\pi n \sqrt{-\Delta}}\right)^{\frac{1}{2}}\left[\frac{2(a+1)}{(1-x)(1+a+b)}\right]^{\frac{n}{2}(a+1)+\frac{\alpha}{2}+\frac{1}{4}}  \\
\begin{aligned}
&\times \left[\frac{2(b+1)}{(1+x)(1+a+b)}
\right]^{\frac{n}{2}(b+1)+\frac{\beta}{2}+\frac{1}{4}} \left[\frac{(1-x^2)(a+b+1)}{4}\right]^{\frac{n}{2}+\frac{1}{4}}  \\
&\times \cos{\left([n(a+1)+\alpha + \frac{1}{2}]\theta+ 
[n(b+1)+\beta + \frac{1}{2}]\gamma-(n+\frac{1}{2})\rho + \frac{\pi}{4}\right)}
\left( 1 + O\left ( {1 \over n} \right ) \right )
\end{aligned}
\end{multline}
valid for general $\alpha, \beta \in \mathbb R$ for $a,b \ge 0$. As noted in
\cite{Co05}, results from \cite{GS91,BG99} imply that the $O(1/n)$ term holds uniformly
in the parameters.
\end{prop}

Actually (\ref{21}) differs from the form reported in \cite{CI91}, with our $\sqrt{-\Delta}$
in the denominator of the first term on the RHS, whereas it is in the numerator of the
corresponding term in \cite{CI91}, and furthermore some signs and factors of 2 in the
cosine are in disagreement. One check is to exhibit the symmetry of the Jacobi polynomials
\begin{equation}\label{Pcd}
P_n^{(c,d)}(-x) = (-1)^n P_n^{(d,c)}(x)
\end{equation}
For this we examine the effect on the parameters (\ref{18}) under the mappings
\begin{align}\label{Pcd1}
x &\mapsto -x, & a &\mapsto b, & b &\mapsto a
\intertext{
We see that $\Delta$ is unchanged, while}
\label{Pcd2}
\rho &\mapsto \pi - \rho, & \theta &\mapsto - \gamma, & \gamma &\mapsto - \theta.
\end{align}
Making the substitutions (\ref{Pcd1}), (\ref{Pcd2}) in (\ref{21}), along with
$\alpha \mapsto \beta$, $\beta \mapsto \alpha$ we see that indeed the RHS is consistent
with (\ref{Pcd}). Another check is to specialize to the case $a=b=0$. Then with
$x = \cos \phi$, $0 \le \phi \le \pi$ we can check from (\ref{18}) that
$\sqrt{-\Delta} = 2 \sin \phi$, $\rho = \pi/2$, $\theta = \pi/2 - \phi/2$,
$\gamma = - \phi/2$ and so
$$
P_n^{(\alpha,\beta)}(\cos \phi) \sim \Big ( {1 \over \pi n} \Big )^{1/2}
{1 \over (\sin \phi/2)^{\alpha + 1/2} (\cos \phi/2)^{\beta + 1/2} }
\cos \Big ( (n + (\alpha + \beta + 1)/2) \theta - (\alpha + 1/2) \pi/2 \Big )
$$
which agrees with the result in Szeg\"o's book \cite{Sz75}. We give our working in
the Appendix.

According to (\ref{final}) and (\ref{5.59}) the particular Jacobi polynomials appearing in the
summation specifying $K$ in (\ref{cK}) depends on the range of values of the continuum line
label $S$. Let us suppose that $1 \le S \le k + 1$, and so the number of particles on
each line is $p$. In this case, and with $s_0 \ge t_0$
\begin{equation}\label{Kps}
K(pS + s_0,y; pS + t_0,x) =
\sum_{n=0}^{p-1} c_n(x,y) \tilde{P}_n^{(an+t_0, bn-t_0)}(x)
\tilde{P}_n^{(an+s_0, bn-s_0)}(y)
\end{equation}
where
\begin{equation}
c_n(x,y) = (1 - y)^{bn-s_0} x^{an+t_0}
{(2n+an+bn+1) (n+an + bn)! n! \over (bn+n - s_0)! (an+n+t_0)!}
\end{equation}
and with $w:= n/p$,
\begin{align}\label{ab}
a &:= {S -1 \over w}, & b &= {1 + k - S \over w}.
\end{align}
The expression in the case $s_0 < t_0$ is the same except that the range of the sum
is now over $1$ to $\infty$ instead of 0 to $p-1$. With $x$ and $y$ given by (\ref{6.1a})
and $r(pS)$ therein equal to $p$ (recall (\ref{3.1})), we want to replace the summand by its
large $n$ asymptotic form. Because of the form of $x$ and $y$ in (\ref{6.1a}), the
parameters $\Delta, \rho, \theta, \gamma$ in Proposition \ref{chi}, as they apply to the
Jacobi polynomials in (\ref{5.59}), all have expansions in inverse powers of $1/p$.

\begin{lemma}
Write
\begin{align}\label{xX}
x &\mapsto 1 - 2x & \text{with }\quad  x &= X_S + {X \over p u_S}.
\end{align}
The quantities (\ref{18}) have the large $p$ expansion
\begin{align}\label{6.11a}
 \Delta &= \Delta_0 + \Delta_1 X/p + {\rm O}(1/p^2) \nonumber \\
 \rho &= \rho_0 + \rho_1 X/p + {\rm O}(1/p^2) \nonumber \\
 \theta &= \theta_0 + \theta_1 X/p +  {\rm O}(1/p^2) \nonumber \\
 \gamma &= \gamma_0 + \gamma_1 X/p +  {\rm O}(1/p^2)
\end{align}
where
\begin{align}\label{6.12}
\Delta_0 = & 4\left(a^2(1-X_S)^2 -X_S(2a(2+b)(1-X_S) +b^2X_S-4b(1-X_S)-4(1-X_S))\right)\nonumber \\
\Delta_1 = & \frac{8}{u_S}\left((4+b^2)X_S -2-a^2(1-X_S)-a(2+b)(1-2X_S)-2b(1-2X_S)\right)\nonumber \\
e^{i \rho_0} = & \frac{2a(1-X_S)-2bX_S+i\sqrt{-\Delta_0}}{ 4\sqrt{(1+a+b)X_S(1-X_S)}}\nonumber \\
\rho_1 = & \frac{1}{\sqrt{-\Delta_0}}\frac{a-(a-b)X_S}{u_SX_S(1-X_S)}\nonumber \\
 e^{i\theta_0} = & \frac{2a+2(a+b+2)X_S+i\sqrt{-\Delta_0}}{4\sqrt{(a+1)(1+a+b)X_S}} \nonumber \\
\theta_1 = & \frac{1}{\sqrt{-\Delta_0}}\frac{a(1-X_S)-(2+b)X_S}{u_SX_S} \nonumber \\
e^{i\gamma_0} = & \frac{ 4\sqrt{(b+1)(1+a+b)(1-X_S)}}{2a+4b+4-2(a+b+2)X_S-i\sqrt{-\Delta_0}} \nonumber \\
\gamma_1 = & \frac{1}{\sqrt{-\Delta_0}}\frac{(2+b)X_S-a(1-X_S)-2}{u_S(1-X_S)}. 
\end{align}
\end{lemma}
\begin{proof}
 According to (\ref{18}), the expansions of $\rho, \theta, \gamma$ depend on the
expansion $\Delta$, so the latter must be done first. For this, all that is required is to substitute
(\ref{xX}) in its definition and perform elementary manipulations. Substituting the expansion of
$\Delta$ into the definitions of 
$\rho, \theta, \gamma$ the values of $\rho_0, \theta_0, \gamma_0$ can be determined immediately.
Making use of these, the stated values of $\rho_1, \theta_1, \gamma_1$ then follow.
\end{proof}

The use of this result is that it allows the terms in the 
summation~(\ref{Kps}) to be exhibited to have a Riemann sum form. 
The idea is to replace the discrete summation label~$n$ by a continuous 
index $1/w$ specifying the
proportionality of the summation label to the large parameter $p$.

\begin{prop}\label{p6.3}
Let $x,y$ be as in (\ref{6.1a}), and let $w_0 < w \le 1$ correspond 
to $\Delta(w) < 0$. 
Let  $w$ be fixed and $p$, $n$ large.
To leading order in contribution to the summation, we can replace
$$
c_n(x,y) \tilde{P}_n^{(an+t_0, bn - t_0)}(x)  \tilde{P}_n^{(an+s_0, bn - s_0)}(y)
$$
with 
\begin{multline}\label{6.13}
\frac{2(a+b+2)}{\pi \sqrt{-\Delta_0}} n^{s_0-t_0} 
\left(\frac{x}{y}\right)^{\frac{an}{2}}\left(\frac{1-y}{1-x}\right)^{\frac{bn}{2}} \\
 \times  
\Real \left[\exp\left(\frac{iw(X-Y)\sqrt{-\Delta_0}}{4u_SX_S(1-X_S)}\right)
\left(\frac{2a(1-X_S)+2bX_S+i\sqrt{-\Delta_0}}{4X_S(1-X_S)}\right)^{s_0-t_0}\right].
\end{multline}
For $0 < w < w_0$, corresponding to $\Delta > 0$, 
to the same order the summand can be replaced by zero. 
\end{prop}
\begin{proof}
 Use of Stirling's formula gives
\begin{equation}\label{cn}
c_n(x,y) 
\sim (1-y)^{bn-s_0}x^{an+t_0} wp(a+b+2)\frac{(n+bn)^{s_0}}{(n+an)^{t_0}}
\frac{(a+b+1)^{n(a+b+1)+\frac{1}{2}}}{(a+1)^{n(a+1)+\frac{1}{2}}(b+1)^{n(b+1)+\frac{1}{2}}}.
\end{equation}
Note that $a$ and $b$ depend on $w$, but not on $n$. 
Furthermore,  we read off from~(\ref{21}) that, 
provided $\Delta < 0$ and thus $w_0 < w \le 1$. 
The assumptions in Proposition~\ref{chi} are therefore satisfied
and we obtain
\begin{multline}\label{6.15}
 \tilde{P}_n^{(an+t_0, bn-t_0)}(x) \tilde{P}_n^{(an+s_0, bn-s_0)}(y) 
 \sim  \frac{4}{\pi n}
\left(\frac{1}{\sqrt{\Delta(x)\Delta(y)}}\right)^{\frac{1}{2}}\\
 \begin{aligned}
&\times \left[\frac{(a+1)}{(1+a+b)}\right]^{n(a+1)+\frac{(s_0+t_0)}{2}+\frac{1}{2}}
 \left[\frac{(b+1)}{(1+a+b)}\right]^{n(b+1)-\frac{(s_0+t_0)}{2}+\frac{1}{2}} 
\\
&\times\left(a+b+1\right)^{n+\frac{1}{2}} \left(xy\right)^{-\frac{n}{2}(a+1)-\frac{1}{4}} \left[(1-x)(1-y)\right]^{-\frac{n}{2}(b+1)-\frac{1}{4}}  \\
 &\times  \left[xy(1-x)(1-y)\right]^{\frac{n}{2}+\frac{1}{4}} \cos{A_x}\cos{A_y}
\end{aligned}
 \end{multline}
where
\begin{eqnarray}\label{6.16}
&&A_x := [n(a+1)+t_0 + \frac{1}{2}]\theta_x+ [n(b+1) -t_0 + 
\frac{1}{2}]\gamma_x -(n+\frac{1}{2})\rho_x + \frac{\pi}{4} \nonumber \\
&&A_y := [n(a+1)+s_0 + \frac{1}{2}]\theta_y+ [n(b+1) -s_0 + 
\frac{1}{2}]\gamma_y -(n+\frac{1}{2})\rho_y + \frac{\pi}{4}.
\end{eqnarray}
We are justified in ignoring the terms O$(1/p)$ in (\ref{21}) due to 
it being a uniform bound. Combining (\ref{cn}) and (\ref{6.15}) shows
\begin{multline}\label{6.17}
  c_nP_n^{(an+t_0, bn-t_0)}(x)P_n^{(an+s_0, bn-s_0)}(y)  \\
 \sim \frac{4(a+b+2)}{\pi \sqrt{-\Delta_0}} n^{s_0-t_0} 
(a+1)^{\frac{s_0-t_0}{2}}(b+1)^{\frac{s_0-t_0}{2}}
\left(\frac{x}{y}\right)^{\frac{an}{2}}
\left(\frac{1-y}{1-x}\right)^{\frac{bn}{2}}  
\cos{(A_x)}\cos{(B_y)}
\end{multline}

Next we make use of the trigonometric identity $\cos{(A)}\cos{(B)} = \frac{1}{2}(\cos{(A+B)}+\cos{(A-B)})$.
The term $\cos{(A+B)}$ oscillates with frequency proportional to $p$ so contributes at a lower order
to the summation than the term $\cos(A-B)$ which has argument of order unity. Now from
(\ref{6.11a}) and (\ref{6.16})
\begin{equation}
\cos(A_x - A_Y) \sim  \cos{\left(w(X-Y)(\theta_1(a+1)+\gamma_1(b+1)-\rho_1) - 
(s_0 -t_0) (\theta_0-\gamma_0)\right)} 
\end{equation}
which using (\ref{6.12}) can be written
\begin{equation}\label{r4}
\cos(A_x - A_Y) \sim  {\rm Re}
\left(\exp\left[\frac{iw(X-Y)\sqrt{-\Delta_0}}{4uX_s(1-X_s)}\right]
e^{i(s_0-t_0)(\theta_0-\gamma_0)}\right)
\end{equation}
Further, we have from (\ref{6.12}) that
\begin{equation}\label{r5}
e^{i(s_0-t_0)(\theta_0-\gamma_0)}  =  
\left(\frac{1}{(a+1)(b+1)(1-X_S)X_S}\right)^{\frac{s_0-t_0}{2}}
\left(\frac{2a(1-X_S)+2bX_S+i\sqrt{-\Delta_0}}{4}\right)^{s_0-t_0} 
\end{equation}
Substituting (\ref{r5}) in (\ref{r4}), multiplying by $1/2$ and substituting in (\ref{6.17})
gives (\ref{6.13}). The summand for $0 < w \le 1$ contributes to a lower order because with
$\Delta > 0$ the Jacobi polynomials are exponentially small \cite{CI91,Iz07}.
\end{proof}

This result in turn allows
the leading large $p$ behaviour of the sum in (\ref{Kps}) and
thus the limit of the scaled correlation kernel (\ref{cK}) to be computed.

\begin{prop}\label{p6p}
Let
$ \bar{K}(s_0,Y;t_0,X)$ be as specified by (\ref{cK}), and let $S$ be such that
$1 \le S \le k+1$ so that $r(pS) = p$. Introduce too the notation
\begin{equation}\label{nu}
\nu := \frac{2+k}{k} \sqrt{ \frac{1+k}{S(2+k-S)}}.
\end{equation}
We have
\begin{equation}\label{KXY}
\bar{K}(s_0,Y;t_0,X) = {\mathcal A}^{X-Y} {\mathcal B}^{s_0 - t_0} K^*(s_0,Y;t_0,X)
\end{equation}
where
\begin{align}\label{AB}
\mathcal{A} &= e^{\pi \nu}, &\mathcal{B}& =  \frac{(2+k)^2kS (2 - k - S)}{4+8k+k^3(1+S)+k^2(5+2S-S^2)}
\end{align}
and
\begin{equation}\label{Kss}
K^*(s_0,Y;t_0,X) = \begin{cases} \displaystyle 
\int_0^1 \Real \left ( e^{\pi i t (X-Y)} (1 + i t \nu)^{s_0 - t_0} \right ) \, dt, & s_0 \ge t_0, \\
 \displaystyle
- \int_1^\infty \Real \left ( e^{\pi i t (X-Y)} (1 + i t \nu)^{s_0 - t_0} \right) \, dt, & s_0 < t_0.
\end{cases}
\end{equation}
\end{prop}
\begin{proof}
It follows from Proposition \ref{p6.3} that
\begin{equation}\label{34}
\sum_{n=0}^{p-1} c_nP_n^{(an+t_0, bn-t_0)}(x)P_n^{(an+s_0, bn-s_0)}(y) \sim \frac{2p^{s_0-t_0+1}}{\pi}
\left(\frac{x}{y}\right)^{\frac{(S-1)p}{2}}\left(\frac{1-y}{1-x}\right)^{\frac{(1+k-S)p}{2}} {\mathcal I}
\end{equation}
where, with $w_0$ such that $-\Delta_0 = 0$,
$$
{\mathcal I} := \int_{w_0}^1\frac{a+b+2}{\pi\sqrt{-\Delta_0}}{\rm Re}
\left[\exp\left(\frac{iw(X-Y)\sqrt{-\Delta}}{4uX_S(1-X_S)}\right)\left(\frac{2aw(1-X_S)+2bwX_S+
iw\sqrt{-\Delta_0}}{4X_S(1-X_S)}\right)^{s_0-t_0}\right] dw
$$
Recalling the dependence of $a$ and $b$ on $w$ from (\ref{ab}), and making the substitution
$$
t = \frac{w\sqrt{-\Delta_0}}{4\pi X_S(1-X_S)u_S}
$$
gives
\begin{equation}\label{36}
{\mathcal I} = \int_0^1 \frac{\pi u}{2} {\rm Re}\left[\exp(\pi i t 
(X-Y))\left(\frac{(2+k)^2(k^2S+2iGt+k(2S-S^2+iGt))}{4+8k+k^3(1+S)+k^2(5+2S-S^2)}\right)^{s_0-t_0}\right] dt
\end{equation}
where $G = \sqrt{S(2+k-S)(1+k)}$. Substituting (\ref{36}) in (\ref{34}) and recalling
(\ref{Kps}) we obtain, after some further minor manipulations, the result (\ref{KXY}) in the
case $s_0 \ge t_0$.  
As noted below (\ref{ab}), in the case $s_0 < t_0$, instead of (\ref{Kps}) we have
$$
K(pS + s_0,y; pS + t_0,x) =
- \sum_{n=p}^{\infty} c_n(x,y) \tilde{P}_n^{(an+t_0, bn-t_0)}(x)
\tilde{P}_n^{(an+s_0, bn-s_0)}(y)
$$
Thus, up to the minus sign, the asymptotic form is again given by (\ref{34}), but with the terminal
of integration in $\mathcal I$ now from 1 to $\infty$. This gives (\ref{KXY}) for
$s_0 > t_0$.
\end{proof}

\begin{cor}
With the scaled correlation function $\bar{\rho}$ specifed by (\ref{6.1b}), and $K^*$ specified by
(\ref{Kss}) we have
\begin{eqnarray}\label{6.1bz}
&& \bar{\rho}(t_1,X_1;\dots;t_n,X_n) = \det [ K^*(X_i,t_i;X_j,t_j) ]_{i,j=1,\dots,n}.
\end{eqnarray}
\end{cor}
\begin{proof}
In the region $1 \le S \le k+1$, this is an immediate consequence of Proposition
\ref{p6p} (note that the prefactors ${\mathcal A}^{X-Y} {\mathcal B}^{s_0 - t_0}$ in
(\ref{KXY}) cancel out of the determinant). The correlation function in the
region $1+k \le S < 2+k$ follows
from the form in the region $0 < S < 1$ upon making the mappings
\begin{align*}
(X_i, t_i) &\mapsto (-X_i, - t_i), & S &\mapsto 2 + k - S
\end{align*}
which from (\ref{Kss}) is indeed a symmetry of (\ref{6.1b}). To show that (\ref{6.1b})
is valid for $0 < S < 1$ requires deriving the analogue (\ref{p6p}) in this case.
Recalling \ref{final}) and (\ref{5.59}), and with 
$s_0 \ge t_0$, we must then obtain the large $p$ form of
\begin{equation}\label{Kps1}
K(pS + s_0,y; pS + t_0,x) =
\sum_{n=0}^{p-1} \tilde{c}_n(x,y) \tilde{P}_{an + t_0 + n}^{(-an-t_0, bn-t_0)}(x)
\tilde{P}_{an + t_0 + n}^{(-an-s_0, bn-s_0)}(y).
\end{equation}
We have done this, using the same general strategy as for $1+k \le S < 2+k$,
obtaining a result consistent with (\ref{6.1b}).
\end{proof}

In the Introduction it was remarked that the bead process was introduced by
Boutillier \cite{Bo06} as a continuum limit of a dimer model on the
honeycomb lattice. The corresponding scaled correlation was calculated to be of
the form (\ref{6.1bz}) but with $K^*$ replaced by $J_\gamma$, where
\begin{equation}\label{KssJ}
J_\gamma(s_0,Y;t_0,X) =\begin{cases}
 \displaystyle
{1 \over 2 \pi} \int_{-1} ^1  \Big ( e^{ i t (X-Y)} (\gamma + i t \sqrt{1 - \gamma^2} )^{s_0 - t_0} 
\Big ) \, dt, & s_0 \ge t_0, \\
 \displaystyle - {1 \over 2 \pi}
\int_{\mathbb R \ [-1,1]}  \Big ( e^{ i t (X-Y)} (\gamma + i t  \sqrt{1 - \gamma^2})^{s_0 - t_0} 
\Big ) \, dt, & s_0 < t_0.
\end{cases}
\end{equation}
The parameter $\gamma$, $|\gamma| < 1$, represents an anisotropy in the underlying $abc$-hexagon, with
$\gamma = 0$ corresponding to the symmetrical case.

We observe that the factor $ (\gamma + i t \sqrt{1 - \gamma^2} )^{s_0 - t_0}$ in the integrands
of (\ref{KssJ}) can be replaced by $(1 + i t  \sqrt{1 - \gamma^2}/\gamma)^{s_0 - t_0}$ without
changing the value of the determinant. Changing scale $J_\gamma(s_0,Y;t_0,X) \mapsto
\pi J_\gamma(s_0,\pi Y;t_0, \pi X)$, and comparing the resulting form of (\ref{KssJ}) with (\ref{Kss}) 
shows the two results to be the same, upon the identification
\begin{equation}\label{J1}
\frac{\sqrt{1 - \gamma^2}}{\gamma} =
\frac{2+k}{k} \sqrt{\frac{1+k}{S(2+k - S)}},
\end{equation}
although the quantity on the RHS is always positive. This latter feature is a consequence of the
calculations relating to our finitized bead process being carried out under the assumption that
$q \ge p$. From the symmetry of the hexagon, the case $q < p$ is obtained by simply replacing
$X,Y$ in (\ref{KssJ}) by $-X, -Y$, or equivalently replacing $\nu$ in   (\ref{Kss}) by $-\nu$.
This then allows us to extend (\ref{J1}) to the region $-1 < \gamma < 0$, by replacing the
$\gamma$ in the denominator by $|\gamma|$.

\section*{Acknowledgements}
The work of the first two  authors have been supported by
a Melbourne Postgraduate Research Award and the Australian 
Research Council respectively. EN has been supported by 
the G\"oran Gustafsson Foundation and by grant 
KAW 2005.0098 from the Knut and Alice Wallenberg Foundation. 
PJF thanks C.~Krattenthaler
for the opportunity to participate in the ESI semester on Combinatorics
and Statistical Mechanics during April 2008, where part of this research
was undertaken.


\section*{Appendix}
\renewcommand{\theequation}{A.\arabic{equation}}
\setcounter{equation}{0}
Chen and Ismail \cite{CI91} use the asymptotic method of Darboux (see e.g.~\cite{Sz75})
applied to the generating function
\begin{equation} \label{correlation}
\sum_{n=0}^\infty  P_n^{\alpha+an, \beta + bn}(x) t^n = 
(1+\xi)^{\alpha+1}(1+\eta)^{\beta+1}[1-a\xi-b\eta-(1+a+b)\xi \eta]^{-1} =: f(t),
\end{equation}
where $\xi$ and $\eta$ depend on $x$ and $t$ according to 
\begin{equation} \label{xi}
\xi = \frac{1}{2}(x+1)t(1+\xi)^{1+a}(1+\eta)^{1+b} \, \, \, {\rm and} \, \, \, \eta = \frac{x-1}{x+1}
\xi.
\end{equation}
The basic idea is to identify and analyze the neighbourhood of the $t$-singularities of $f(t)$,
to replace $f(t)$ in (\ref{correlation}) by its leading asymptotic form $g(t)$ in the neighbourhood of
the singularities (referred to as the comparison function), and finally to expand the latter
about the origin to equate coefficients of $t^n$ and so read off the asymptotic form of
$ P_n^{\alpha+an, \beta + bn}(x)$.

It is shown in \cite{CI91} that the $t$-singularities of $f(t)$ occur at
\begin{equation}\label{tpm}
t_\pm  =  \frac{b(x-1)+a(x+1) \pm i \sqrt{- \Delta}}{(1+a+b)(1-x^2)}[1+\xi_\pm]^{-1-a}[1+\eta_\pm]^{-1-b}
\end{equation}
where $\Delta$ is given by (\ref{18}) and
\begin{align}\label{xpm}
\xi_\pm  &=  \frac{b(x-1)+a(x+1) \pm i \sqrt{\Delta}}{2(1+a+b)(1-x)}, &
 \eta_\pm &= \frac{x-1}{x+1}\xi_\pm
\end{align}
The comparison function is computed as
\begin{equation}\label{gt}
g(t) = B_+ (t_+ - t)^{-1/2} + B_- (t_1 - t)^{-1/2}
\end{equation}
where
\begin{equation}\label{Bf}
B_{\pm} = \lim_{t \to t_\pm} (t_{\pm} - t)^{1/2} f(t).
\end{equation}
Since we know from \cite{CI91} that
$$
{d t \over d \xi} \Big |_{\xi = \xi_{\pm}} = 0, \qquad
{d^2 t \over d \xi^2} \Big |_{\xi = \xi_{\pm}} =
\frac{\pm 2 i \sqrt{-\Delta}}{(1+x)^2(1+\xi_\pm)^{2+a}(1+\eta_\pm)^{2+b}} =: 2 A_{\pm},
$$
we calculate from (\ref{Bf}) that
\begin{align}\label{fu}
B_\pm  = & \mp \frac{(1 + \xi_\pm)^{\alpha+1}(1 + \eta_+)^{\beta + 1} (x+1) \sqrt{- A_{\pm}}}{
i \sqrt{ - \Delta}}  \\
& =  e^{\mp \pi i /4} (- \Delta)^{-1/4}(1 + \xi_{\pm})^{\alpha - a/2} (1 + \eta_\pm)^{\beta - b/2}\nonumber
\end{align}
The first of the formulas in (\ref{fu}) is not reported in \cite{CI91}, while the second is their
(2.14), (2.15) but with our $ e^{\mp \pi i /4} (- \Delta)^{-1/4}$ replaced by
$- i (\Delta)^{-1/4}$ and $i (- \sqrt{\Delta})^{-1/2}$ respectively.

The coefficient of $t^n$ in (\ref{gt}), and thus the leading asymptotic form of
$ P_n^{\alpha+an, \beta + bn}(x)$ according to the method of Darboux, is equal to
\begin{equation}\label{gt1}
(-1)^n {-\frac{1}{2} \choose n}[B_+t_+^{-n-\frac{1}{2}}+ B_-t_-^{-n-\frac{1}{2}}]
\end{equation}
To simplify this, we note from (\ref{18}) and (\ref{xpm}) that
\begin{align*}
 1 + \xi_\pm &= \Big ( \frac{2(a+1)}{(1-x)(a+b+1)} \Big )^{1/2} e^{\pm i \theta} \\
1 + \eta_\pm &= \Big ( \frac{2(b+1)}{(1+x)(a+b+1)} \Big )^{1/2} e^{\pm i \gamma} \\
 \frac{2 \xi_\pm}{x+1} &= \frac{2 e^{\pm i \rho}}{\sqrt{(a+b+1)(1-x^2)}}
\end{align*}
These substituted in (\ref{tpm}) and (\ref{fu}) give
\begin{equation*}
\begin{split}
 B_+ t_+^{-n-\frac{1}{2}}+ B_-t_-^{-n-\frac{1}{2}} = 
 &\left(\frac{1}{\sqrt{-\Delta}}\right)^{\frac{1}{2}}
2 \cos h(\theta,\gamma,\rho) 
\left[\frac{2(a+1)}{(1-x)(1+a+b)}
\right]^{\frac{n}{2}(a+1)+\frac{\alpha}{2}+\frac{1}{4}} \\
& \times \left[\frac{2(b+1)}{(1+x)(1+a+b)}\right]^{\frac{n}{2}(b+1)+\frac{\beta}{2}+\frac{1}{4}} 
\left[\frac{(1-x^2)(a+b+1)}{4}\right]^{\frac{n}{2}+\frac{1}{4}} 
\end{split}
\end{equation*}
where
$$
h(\theta, \gamma, \rho) = [n(a+1)+\alpha + \frac{1}{2}]\theta+ [n(b+1)+\beta + \frac{1}{2}]i
\gamma-(n+\frac{1}{2})\rho + \frac{\pi}{4}
$$
This, together with the expansion
$$
(-1)^n{-\frac{1}{2} \choose n} \sim \frac{n^{-\frac{1}{2}}}{\sqrt{\pi}}
$$
substituted in (\ref{gt1}) gives (\ref{21}).

 \bibliography{FFN08}{}
\bibliographystyle{alpha}

\end{document}